\documentclass[11pt,reqno]{amsart}
\usepackage[tmargin=1in,bmargin=1in,rmargin=1in,lmargin=1in]{geometry}
\usepackage[mathscr]{eucal}
\usepackage[breaklinks=true]{hyperref}
\usepackage[dvipsnames]{xcolor}
\usepackage{enumerate}

\usepackage{amsmath,amsthm,amsfonts,amssymb,tocvsec2}
\usepackage{mathdots}
\usepackage{mathtools}
\usepackage{caption}
\usepackage{subcaption}
\usepackage{float}
\usepackage{multicol}
\usepackage{enumitem}

\theoremstyle{plain}
\newtheorem{theorem}{Theorem}[section]

\newtheorem{lemma}[theorem]{Lemma}
\newtheorem{prop}[theorem]{Proposition}

\newtheorem{utheorem}{\textrm{\textbf{Theorem}}}

\theoremstyle{definition}
\newtheorem{defn}[theorem]{Definition}

\newtheorem{rem}[theorem]{Remark}

\numberwithin{equation}{section}

\begin{document}
	\title[ Sign regularity preserving linear operators
	]{Sign regularity preserving linear operators}
	
	\author{Projesh Nath Choudhury and Shivangi Yadav}
	\address[P.N.~Choudhury]{Department of Mathematics, Indian Institute of Technology Gandhinagar, Gujarat 382355, India}
	\email{\tt projeshnc@iitgn.ac.in}
	\address[S.~Yadav]{Department of Mathematics, Indian Institute of Technology Gandhinagar, Gujarat 382355, India; Ph.D. Student}
	\email{\tt shivangi.yadav@iitgn.ac.in, shivangi97.y@gmail.com}
	
	\date{\today}
	
	\begin{abstract}
		A matrix $A\in \mathbb{R}^{m \times n}$ is strictly sign regular/SSR (or sign regular/SR) if for each $1 \leq k \leq \min\{m,n\}$, all (non-zero) $k\times k$ minors of $A$ have the same sign. This class of matrices contains the totally positive matrices, and was first studied by Schoenberg in 1930 to characterize variation diminution, a fundamental property in total positivity theory. In this article, we classify all surjective linear mappings $\mathcal{L}:\mathbb{R}^{m\times n}\to\mathbb{R}^{m\times n}$ that preserve: (i)~sign regularity and (ii)~sign regularity with a given sign pattern, as well as (iii)~strict versions of these.
	\end{abstract}
	
	
	
	\maketitle
	
	\vspace*{-9mm}
	\settocdepth{section}
	\tableofcontents

\section{Introduction and main results} \label{section introduction}

This work studies preserver problems for totally positive matrices, in
fact for the broader class of (strictly) sign regular matrices. Preserver problems have been the focus of considerable work in analysis and operator theory, over the past century and even earlier.

We will focus on \textit{linear preserver problems}, which aim to describe the general form of linear transformations on a space of bounded linear operators that leave certain functions, subsets, relations, etc.\ invariant. The primary notion of a preserver that we study in this paper is as follows: given a subset $S$ of the space $V$ consisting of bounded linear operators on a Banach space, a linear operator $\mathcal{L}:V\to V$ such that $\mathcal{L}(S)=S$ is called a \textit{linear preserver of $S$} or \textit{$S$-preserver}.
In 1897, Frobenius~\cite{Frobenius1897} characterized the general form of all determinant preserving linear maps on matrix algebras, which is regarded as the first result on linear preserver problems. These problems have since been widely studied in matrix theory and operator theory. Some examples include the problem of spectrum preserving transformations on a space of bounded linear operators over a Banach space (Jafarian--Sourour~\cite{Jaf-Sour86}), linear preservers of the unitary group in $\mathbb{C}^{n\times n}$ (Marcus~\cite{Marcus59}) and on arbitrary $C^*$-algebras (Russo--Dye~\cite{Russo-Dye66}), 
linear transformations on operator algebras preserving absolute values
(Radjabalipour--Seddighi--Taghavi~\cite{Radja-Sedd-Tag01}), and linear
maps preserving invertibility (Sourour~\cite{Sour96}). For more details about linear preserver problems and some techniques to tackle them, we refer to~\cite{GLS2000,Li-Pierce01,Molnar07}. The classification of linear preservers of various forms of positivity has long been studied and is an important problem in the preserver literature -- for instance, the linear preservers of copositive matrices and completely positive rank preserving linear maps  have been recently characterized by Shitov~\cite{Shi21,Shi23}. However,  the linear preservers of positive semidefinite matrices have been actively studied but are  still not completely classified.

The focus of this work is on exploring linear preservers of a class of matrices that includes as a special case the widely studied totally
positive matrices. Given integers $m,n\geq k\geq1$, a matrix $A \in \mathbb{R}^{m\times n}$ is said to be \textit{strictly sign regular of order $k$} (SSR$_k$) if for all $1 \leq r \leq k$, there exists a sequence of signs $\epsilon_r\in\{1,-1\}$ such that every $r\times r$ minor of $A$ has sign $\epsilon_r$. An SSR$_k$ matrix $A$ is \textit{strictly sign regular} (SSR) if $k=\mathrm{min}\{m,n\}$. If minors are allowed to also vanish, then $A$ is correspondingly said to be \textit{sign regular of order $k$} (SR$_k$) and \textit{sign regular} (SR). For an SSR (respectively SR) matrix $A$, if $\epsilon_r=1$ for all $r\geq1$ then $A$ is \textit{totally positive} (TP) (respectively \textit{totally non-negative} (TN)). These matrices arise in diverse sub-fields of mathematics, including analysis, approximation theory, cluster algebras, combinatorics, differential equations, Gabor analysis, integrable systems, interpolation theory and splines, matrix theory, probability and statistics, Lie theory, and representation theory \cite{Ando87,BGKP20,Bre95,fallat-john,FZ02,gantmacher-krein,GRS18,Karlin64,K68,Karlinsplines,KW14,Lu94,pinkus,Ri03,S55,Whitney}.

A fundamental property widely linked with SSR and SR matrices (including TP/TN matrices) is variation diminution (VD) which says that if a matrix $A$ acts on a vector ${\bf x}$ then the number of sign changes in the coordinates of $A{\bf x}$ is at most this
number for ${\bf x}$~\cite{BJM81,C22,S30}. This was first studied by Fekete in correspondence with P\'{o}lya~\cite{FP12} using P\'{o}lya frequency sequences. In 1930, Schoenberg showed that the VD property holds for SR matrices, followed by Motzkin~\cite{Mot36} (1936) who characterized all $m \times n$ matrices of rank $n$ with the VD property, by showing that they are precisely the sign regular matrices. These progresses were taken forward by Gantmacher--Krein~\cite{GK50} who characterized $m \times n$ strictly sign regular matrices with the size restriction $m>n$ using the variation diminishing property. In prior recent work~\cite{CY-VDP23}, we concluded this line of investigation by (i)~characterizing all $m \times n$ SSR matrices using variation diminution, (ii)~characterizing strict sign regularity of a given sign pattern in terms of VD, and (iii)~strengthening Motzkin's characterization of sign regular matrices by omitting the rank constraint and specifying the sign pattern.

We now come to the question of classifying the linear preservers of SSR and SR matrices. For the restricted  subclasses of TP and TN matrices, which are moreover square, their  linear preservers had been classified by Berman--Hershkowitz--Johnson~\cite{BHJ85}. In this paper, we extend their work in several ways. First, the results below hold for arbitrary sizes. Secondly, we classify the linear preservers of SSR and SR matrices (allowing all sign patterns). Third, we also do this for every fixed sign pattern. We need the following definitions and notations to state our main results.

\begin{defn} Throughout this manuscript, let $m,n\geq1$ be integers.
	\begin{enumerate}
		\item An \textit{exchange matrix} is a square antidiagonal $0$-$1$ matrix of the form
		\[P_n:=\begin{pmatrix}
			0 & \cdots & 1 \\
			\vdots & \iddots & \vdots \\
			1 & \cdots & 0 \\
		\end{pmatrix} \in \{ 0, 1 \}^{n\times n}.\]
		
		\item For an SSR (SR) matrix $A \in \mathbb{R}^{m\times n}$, its \textit{sign pattern} is the
		ordered tuple $\epsilon = (\epsilon_1, \dots, \epsilon_{\min\{m,n\}})$, with $\epsilon_r$ denoting the sign of all (non-zero) $r \times r$ minors of $A$. 
		
		\item Given integers $m,n\geq k\geq 1$ and a sign pattern $\epsilon=(\epsilon_1,\ldots,\epsilon_k)$, an SSR$_k$ (SR$_k$) matrix $A\in \mathbb{R}^{m\times n}$ is called SSR$_k({\epsilon})$ (SR$_k({\epsilon})$) if the sign pattern of the (non-zero) minors of $A$ of order at most $k$ is given by $\epsilon$. If $k=\min\{m,n\}$ then we simply write SSR$(\epsilon)$ (SR$(\epsilon)$).
		
		\item Let $\mathcal{SR}$ denote the class of SR matrices of a given fixed size. Similarly we define $\mathcal{SR}_2(\epsilon)$, $\mathcal{SR}(\epsilon)$, $\mathcal{SSR}$, and $\mathcal{SSR}(\epsilon)$, where $\mathcal{SR}_2(\epsilon)$ is only concerned with  the signs $\epsilon_1$ and $\epsilon_2$.
		
		\item Let $P(S)$ denote the set of $S$-preservers, for $S$ among $\mathcal{SR},\ \mathcal{SR}_2,\ \mathcal{SR}(\epsilon),\ \mathcal{SR}_2(\epsilon)$.
	\end{enumerate} 
\end{defn}

We now state our main results. Our first theorem characterizes all linear preservers for the class of SR matrices. Moreover, we show that surprisingly, to classify the linear preservers of $\mathcal{SR}$ it suffices to examine the linear preservers of $\mathcal{SR}_2$.  

\begin{utheorem} \label{D}
	Let $\mathcal{L}:\mathbb{R}^{m\times n}\to\mathbb{R}^{m\times n}$  be a linear transformation, where $m,n \geq 2$ are integers such that $\max\{m,n\}\geq3$. Then the following statements are equivalent.
	\begin{itemize}
		\item[(1)] $\mathcal{L}$ maps the class of $m\times n$ SR matrices onto itself.
		\item[(2)] $\mathcal{L}$ maps the class of $m\times n$ SR$_2$ matrices onto itself.
		\item[(3)] $\mathcal{L}$ is a composition of one or more of the following types of transformations:
		\begin{itemize}
			\item[(a)] $A\mapsto FAE$, where $F_{m\times m}, E_{n\times n}$ are diagonal matrices with positive diagonal entries;
			\item[(b)] $A\mapsto -A$;
			\item[(c)] $A\mapsto P_mA$, in which $P_m$ is an exchange matrix;
			\item[(d)] $A\mapsto AP_n$;  and
			\item[(e)] $A\mapsto A^T$, provided $m=n$.
		\end{itemize}
	\end{itemize}
	Moreover, the theorem is also true if SR is replaced by SSR in parts (1) and (2).
\end{utheorem}

Note that if $m=n$ then (3)(d) is not needed, given (c) and (e).

\begin{rem}
	If $m \neq n$ and $\min\{m,n\}=1$ in Theorem~\ref{D}, the problem reduces to classifying linear $\mathcal{SR}_1$-preservers. In this case Theorem~\ref{D} still holds, but now the second statement is $\mathcal{L} \in P (\mathcal{SR}_1)$ in place of  $\mathcal{L} \in P(\mathcal{SR}_2)$, and $P_m, P_n$ can be any permutation matrices instead of exchange matrices in the third statement. If $m=n=2$, there is only one $2\times 2$ minor and since we are considering  the set $S$ of all sign patterns of SR matrices in Theorem~\ref{D}, the problem again reduces to classifying linear $\mathcal{SR}_1$-preservers. Our second main result addresses this latter case.
\end{rem}

\begin{utheorem}\label{thrmSR_1}
	Given a linear map $\mathcal{L}:\mathbb{R}^{2\times2}\to\mathbb{R}^{2\times 2}$, the following are equivalent.
	\begin{itemize}
		\item[(1)] $\mathcal{L}$ maps the class of $2\times 2$ SR matrices onto itself.
		\item[(2)] $\mathcal{L}$ maps the class of $2\times 2$ SR$_1$ matrices onto itself .
		\item[(3)] $\mathcal{L}$ is a composition of one or more of the following types of transformations:
		\begin{itemize}
			\item[(a)] $A\mapsto H\circ A$, where $H$ is an entrywise positive matrix and $\circ$ denotes the Hadamard product;
			\item[(b)] $A\mapsto -A$;
			\item[(c)] $A\mapsto P_2A$, in which $P_2$ is an exchange matrix;
			\item[(d)] $A\mapsto A^T$; and
			\item[(e)] $\begin{pmatrix}
				a_{11} & a_{12} \\
				a_{21} & a_{22}
			\end{pmatrix}\mapsto\begin{pmatrix}
				a_{11} & a_{12} \\
				a_{22} & a_{21}
			\end{pmatrix}$.
		\end{itemize}
	\end{itemize}
	Moreover, the theorem is also true if SR$_1$ is replaced by SSR$_1$ in part (2).
\end{utheorem}

In this case, we additionally obtain the map specified in (e), whose higher-dimensional analogue does not appear to exist. Moreover, this map does not preserve the invertibility of a $2\times2$ matrix. Therefore, we classify the linear preservers for the class of $2\times2$ SSR matrices in Theorem~\ref{Theorem_SSR2}.
	
Next, observe that the two theorems above do not assure that the sign patterns of an SR or SSR matrix $A$ and its image $\mathcal{L}(A)$ will be identical. In our final main result, our objective is to characterize linear preservers for the class of SR($\epsilon$) matrices for any size $m \times n$ and any given sign pattern $\epsilon$. Further, we show that it is again sufficient to study the linear preservers of $\mathcal{SR}_2(\epsilon)$  in order to characterize linear $\mathcal{SR}(\epsilon)$-preservers.

\begin{utheorem}\label{F}
	Let $\epsilon$ be a given sign pattern and $\mathcal{L}:\mathbb{R}^{m\times n}\to\mathbb{R}^{m\times n}$ be a linear transformation, where  $m,n\geq2$. Then the following statements are equivalent.
	\begin{itemize}
		\item[(1)] $\mathcal{L}$ maps the class of $m\times n$ SR$(\epsilon)$ matrices onto itself.
		\item[(2)] $\mathcal{L}$ maps the class of $m\times n$ SR$_2(\epsilon)$ matrices onto itself.
		\item[(3)] $\mathcal{L}$ is a composition of one or more of the following types of transformations:
		\begin{itemize}
			\item[(a)] $A\mapsto FAE$, where $F_{m\times m}, E_{n\times n}$ are diagonal matrices with positive diagonal entries;
			\item[(b)] $A\mapsto P_mAP_n$, where $P_m, P_n$ are exchange matrices; and
			\item[(c)] $A\mapsto A^T$, provided $m=n$.
		\end{itemize}
	\end{itemize}
	Moreover, the theorem is also true if SR$(\epsilon)$ is replaced by SSR$(\epsilon)$.
\end{utheorem}

In particular, by taking $m=n$ and  $\epsilon_k=1$ for all $k$, Theorem~\ref{F} gives the linear preservers for the classes of $n\times n$ TP and TN matrices as a special case which were characterized by
Berman--Hershkowitz--Johnson~\cite{BHJ85}. Thus, the three results above -- along with Theorem \ref{Theorem_SSR2} -- conclude (in a sense) the linear preserver problem for a well studied class of positive matrices.

\begin{rem}
The proofs of Theorems~\ref{D}~and~\ref{F} will reveal that the linear preservers of $\mathcal{SR}$ and $\mathcal{SR}_2$ (respectively $\mathcal{SR}(\epsilon)$ and $\mathcal{SR}_2(\epsilon)$) coincide with the linear preservers of
$\mathcal{SR}_k$ (respectively $\mathcal{SR}_k(\epsilon)$) for $2\leq k\leq\mathrm{min}\{m,n\}$.
\end{rem}

\textbf{Organization of the paper:} The remaining sections of the paper are devoted to proving our main results. In Section~\ref{section SR} we recall some basic results and prove Theorems~\ref{D}~and~\ref{thrmSR_1}, which classify all linear maps that preserve SR/SSR matrices. In the final section, we prove Theorem~\ref{F}.

\section{Theorems~\ref{D}~and~\ref{thrmSR_1}: Linear preserver problem for sign regularity}\label{section SR}

The goal of this section is to prove Theorems~\ref{D}~and~\ref{thrmSR_1}: classify all linear sign regularity preservers. We begin by proving Theorem~\ref{D}. The implications (3)$\implies$(1) and (3)$\implies$(2) follow from straightforward calculations. Note that the map $A\mapsto A^T$ does not preserve the domain $\mathcal{SR}\cap\mathbb{R}^{m\times n}$ for $m\neq n$. So we need to treat the cases $m=n$ and $m\neq n$ separately. We will first show that (2)$\implies$(3), and finally why (1)$\implies$(3) after considering both cases $m=n$ and $m\neq n$. To proceed, we need the following basic notations and two preliminary results.
\begin{itemize}
	\item[(i)] Let $E_{ij}$ denote the matrix whose $(i,j)$ entry is one, and zero otherwise.
	
	\item[(ii)] We define the set $S_{ij}:=\{E_{pq}:\mathcal{L}(E_{ij})_{pq}\neq0\}$.
	
	\item[(iii)] Let $J=J_{m \times n}$ be the $m \times n$ matrix with all entries 1. We will denote $\mathcal{L}(J)$ by $\mathcal{L}(J):=Y=(y_{ij})$, and we will specify the order $m\times n$ as and when we use the matrix $J$.
	
	\item[(iv)] Let $\epsilon_i:=\epsilon_i(A)$ denote the sign of the $i\times i$ non-zero minors of $A\in\mathcal{SR}$.
	
	\item[(v)] We say that a real matrix $A = (a_{ij}) \geq 0$ if each $a_{ij} \geq 0$. Similarly for $A \leq 0$.
	
	\item[(vi)] A real square matrix $A$ is called a \textit{monomial matrix} if all rows and columns of $A$ contain exactly one non-zero entry, which is moreover positive. It is well-known that these are precisely the matrices $A$ such that $A, A^{-1} \geq 0$.
	
	\item[(vii)] Let $A\begin{pmatrix}
		i_1,\ldots,i_p \\
		j_1,\ldots,j_q
	\end{pmatrix}$ denote the $p\times q$ submatrix of an
	$m\times n$ matrix $A$ indexed by rows $1\leq i_1<\cdots<i_p\leq m$ and columns $1\leq j_1<\cdots<j_q\leq n$, where $p\in [1, m], q\in[1,n]$ are integers.
	
	\item[(viii)] We denote by
		\begin{equation}\label{basis}
		\mathfrak{B}=(E_{11},\ldots,E_{kk};E_{12},\ldots,E_{mn})
		\end{equation} a fixed ordered basis of each $\mathbb{R}^{m\times n}$, where $k=\min\{m,n\}$. The basis elements are arranged such that $E_{ii}$ appear first in increasing order of $i$, followed by the remaining elements arranged in lexicographic order.
\end{itemize}
The first preliminary result is a density theorem proved by Gantmacher--Krein in 1950 using the total positivity of the Gaussian kernel.

\begin{theorem}[Gantmacher--Krein \cite{GK50}] \label{G-K}
	Given integers $m,n\geq k\geq1$ and a sign pattern $\epsilon=(\epsilon_1,\ldots,\epsilon_k)$, the $m\times n$ SSR$_k(\epsilon)$ matrices are dense in the $m\times n$ SR$_k(\epsilon)$ matrices.
\end{theorem}

The following preliminary lemma will be handy in classifying linear $\mathcal{SSR}$-preservers using linear $\mathcal{SR}$-preservers. We provide the proof for completeness.

\begin{lemma}\cite[Lemma 1]{BHJ85}\label{lemma_preserver}
	Let $S$ be a subset of a finite-dimensional real vector space $V$. Then $P(S)\subseteq P(\overline{S}) \subseteq P(\mathrm{span}(S))$.
\end{lemma}

\begin{proof}
	Let $\mathcal{L}\in P(S)$. Then $\mathcal{L}$ maps $\mathrm{span}(S)$ homeomorphically onto itself, since dim$V<\infty$. As $\overline{S}\subseteq \mathrm{span}(S)$, it follows that $\mathrm{span}(S)=\mathrm{span}(\overline{S})$, and $\mathcal{L}(\overline{S})=\overline{\mathcal{L}(S)}=\overline{S}$.
\end{proof}
First, we prove certain propositions which will be used in proving Theorem~\ref{D} for both the $m=n$ and $m\neq n$ cases. Let $\mathcal{L}:\mathbb{R}^{m\times n}\to\mathbb{R}^{m\times n}$ be a linear transformation that maps $\mathcal{SR}_2$ onto itself, where $m,n\geq 2$ are integers such that $\max\{m,n\}\geq3$. If $m\neq n$, then we assume without loss of generality that $m>n$, i.e., the number of rows are more than the number of columns. The case when $m<n$ is dealt similarly by pre -- and post -- composing $\mathcal{L}$ with $A\mapsto A^T$.

Since $E_{ij}\in \mathcal{SR}_2$ for all $1\leq i\leq m$, $1\leq j\leq n$, and $\mathcal{L}(\mathcal{SR}_2)=\mathcal{SR}_2$, it follows that $\mathcal{L}$ maps $\mathbb{R}^{m\times n}$ onto itself. Hence $\mathcal{L}^{-1}$ exists and $\mathcal{L}^{-1}\in P(\mathcal{SR}_2)$.

\begin{prop}\label{SR_square_image_same_epsilon_1_sign}
	Let $\mathfrak{B}$ be the ordered basis~\eqref{basis}. Then $\epsilon_1(\mathcal{L}(E_{ij}))$ has the same sign for all $E_{ij}\in \mathfrak{B}$.
\end{prop}

\begin{proof}
	Indeed, suppose that there exist two distinct $E_{ef}$ and $E_{kl}$ in $\mathfrak{B}$ such that $\mathcal{L}(E_{ef})=U=(u_{ij})_{i,j=1}^{m,n}$ and $\mathcal{L}(E_{kl})=V=(v_{ij})_{i,j=1}^{m,n}$ but $\epsilon_1(U)\neq\epsilon_1(V)$. Without loss of generality, assume that $\epsilon_1(U)=1$ and $\epsilon_1(V)=-1$. Therefore, $u_{ij}\geq 0\geq v_{ij}$ for all $1\leq i\leq m$, $1\leq j\leq n$.
	
	Note the following facts about the matrices $U$ and $V$:
	\begin{itemize}
		\item[(i)] $U, V\neq0_{m\times n}$, since $\mathcal{L}$ is invertible.
		
		\item[(ii)] \textit{$u_{st}\neq0$ if and only if $v_{st}\neq0$ where $1\leq s\leq m,$ $1\leq t\leq n$.} Suppose not; then some $u_{st}>0$ but $v_{st}=0$. Note that for all $c\geq0$, $E_{ef}+cE_{kl}\in \mathcal{SR}_2$, and thus, $\mathcal{L}(E_{ef}+cE_{kl})=U+cV\in \mathcal{SR}_2$. Now, $(U+cV)_{st}=u_{st}>0$. Since $V\neq0$ therefore there exist $1\leq q\leq m,$ $1\leq r\leq n$ such that $v_{qr}<0$. Hence, we can choose $c>0$ such that $(U+cV)_{qr}<0$, a contradiction.
		
		\item[(iii)] \textit{There exist at least two non-zero entries in $U$ and $V$.} Suppose instead that $u_{ij}$ is the only non-zero entry in $U$. Again, we can choose $c>0$ such that $\mathcal{L}(E_{ef}+cE_{kl})=0_{m\times n}$, a contradiction.
		
		\item[(iv)] $U\neq\alpha V$ for any $\alpha\in\mathbb{R}$.
	\end{itemize}
	From the above observations, there exist $1\leq i,s\leq m$ and $1\leq j,t\leq n$ such that $u_{ij}, v_{ij}, u_{st}, v_{st}\neq0$ and $\frac{u_{ij}}{v_{ij}}\neq\frac{u_{st}}{v_{st}}$. Without loss of generality, assume that $\frac{-u_{ij}}{v_{ij}}<\frac{-u_{st}}{v_{st}}$. Now choose $c$ such that $0<\frac{-u_{ij}}{v_{ij}}<c<\frac{-u_{st}}{v_{st}}$. For this $c>0$, $\mathcal{L}(E_{ef}+cE_{kl})\in \mathcal{SR}_2$ has a positive and a negative entry, a contradiction. Hence $\epsilon_1(\mathcal{L}(E_{ij}))$ has the same sign for all $i,j$.
\end{proof}

\begin{rem}\label{Ato-Awlog}
	Note that $\mathcal{L}$ is a linear $\mathcal{SR}_2$-preserver if and only if $\mathcal{L}$ composed with the map $A\mapsto-A$ is an $\mathcal{SR}_2$-preserver. So without loss of generality we assume $\epsilon_1(\mathcal{L}(E_{ij}))=1\ \forall~E_{ij}\in\mathfrak{B}$.
\end{rem} 

\begin{prop}\label{SR_square_monomial}
	Let $L\in\mathbb{R}^{mn\times mn}$ be the matrix representation of $\mathcal{L}$ with respect to the ordered basis $\mathfrak{B}$ as in~\eqref{basis}. Then $L$ is a monomial matrix.
\end{prop}

\begin{proof}
	From Remark~\ref{Ato-Awlog}, $\epsilon_1(\mathcal{L}(E_{ij}))=1$ for all $E_{ij}\in\mathfrak{B}$. Thus, $L\geq 0$. Since $\mathcal{L}^{-1}$ is also a linear $\mathcal{SR}_2$-preserver, by Proposition~\ref{SR_square_image_same_epsilon_1_sign} either $L^{-1}\geq 0$ or $L^{-1}\leq 0$. But, $L\geq 0$ and $LL^{-1}=I$, therefore $L^{-1}\geq 0$. Hence $L$ is a monomial matrix.
\end{proof}

\begin{rem}\label{Sij_singleton}
	Since $L$ is a monomial matrix, $S_{ij}$ is a non-empty singleton set for all $1\leq i\leq m$, $1\leq j\leq n$. Also, $S_{ij}\cap S_{kl}=\emptyset$ for $(i,j)\neq(k,l)$.
\end{rem}

We now prove Theorem~\ref{D} for $m=n\geq3$ using Proposition~\ref{SR_square_monomial}.

\begin{proof}[Proof of Theorem~\ref{D}~(2)$\implies$(3) for $m=n$]
	Let $\mathcal{L}:\mathbb{R}^{n\times n}\to\mathbb{R}^{n\times n}$ for $n\geq3$ be a linear  map sending the class of SR$_2$ matrices onto itself. Note that Proposition~\ref{SR_square_monomial} now holds. We now split the proof into several propositions.
	
	\begin{prop}\label{SR_square_S11SnnS1nSn1_elements}
		For $\mathcal{L}\in P(\mathcal{SR}_2)$, the element in each of the sets $S_{11},\ S_{nn},\ S_{1n}$, and $S_{n1}$ must be among the following: $E_{11},\ E_{nn},\ E_{1n},\ \text{or}\ E_{n1}$.
	\end{prop}
	
	\begin{proof}
		First, we will prove the result for $S_{11}$. Suppose that $S_{11}\neq\{E_{11}\}, \{E_{nn}\}, \{E_{1n}\}, \{E_{n1}\}$. Let $J(c)$ be the matrix of size $n\times n$ obtained by multiplying the $(1,1)$ entry of $J=J_{n\times n}$ by $c$. Then $J(c)\in \mathcal{SR}_2$ for $c>0$ and hence $\mathcal{L}(J(c)):=Y(c)\in \mathcal{SR}_2$. Note that all the entries of $Y(c)$ are non-zero by Remark~\ref{Sij_singleton}. Now, consider the following cases.
		
		If $S_{11}=\{E_{1k}\}$ where $k\neq1,n$, consider the following $2 \times 2$ minors of $Y(c)$:
		\[ \det Y(c) \begin{pmatrix} 1,n \\ 1,k \end{pmatrix}
		= \det\begin{pmatrix}
			y_{11} & cy_{1k} \\
			y_{n1} & y_{nk}
		\end{pmatrix}, \qquad
		\det Y(c) \begin{pmatrix} 1,n \\ k,n \end{pmatrix}
		= \det\begin{pmatrix}
			cy_{1k} & y_{1n} \\
			y_{nk} & y_{nn}
		\end{pmatrix}. \]
		One can choose $c>0$ such that these minors have opposite signs, a contradiction. 
		
		Similarly, we can show that $S_{11}\neq\{E_{k1}\},\ \{E_{nk}\},\ \{E_{kn}\}$ for $k\neq1,n$. Now if $S_{11}=\{E_{ij}\}$ where
		$\{i,j\}\neq\{1,n\}$, then we get a contradiction by considering the $2 \times 2$ minors of $Y(c)$:
		\[ \det Y(c) \begin{pmatrix} 1,i \\ 1,j \end{pmatrix}
		= \det\begin{pmatrix}
			y_{11} & y_{1j} \\
			y_{i1} & cy_{ij}
		\end{pmatrix}, \qquad
		\det Y(c) \begin{pmatrix} i,n \\ 1,j \end{pmatrix}
		= \det\begin{pmatrix}
			y_{i1} & cy_{ij} \\
			y_{n1} & y_{nj}
		\end{pmatrix}.\]
		By a similar arguement, one can show this for each of $S_{nn},\ S_{1n}$, and $S_{n1}$ by multiplying the $(n,n),\ (1,n)$, and $(n,1)$ entries, respectively, of $J=J_{n\times n}$ by $c>0$.
	\end{proof}
	
	\begin{prop}\label{SR_square_S11Snn_S1nSn1_combo}
		For $\mathcal{L}\in P(\mathcal{SR}_2)$, the following pairwise combinations are possible.
		\begin{itemize}
			\item[(i)] $S_{11}=\{E_{11}\}\ \text{and} \ S_{nn}=\{E_{nn}\}, \ \text{or} \ S_{11}=\{E_{nn}\} \ \text{and}\ S_{nn}=\{E_{11}\}, \ \text{or}$ \\
			$S_{11}=\{E_{1n}\}\ \text{and}\ S_{nn}=\{E_{n1}\}, \ \text{or}\ S_{11}=\{E_{n1}\} \ \text{and}\ S_{nn}=\{E_{1n}\}$.
			
			\item[(ii)] $S_{1n}=\{E_{1n}\} \ \text{and}\ S_{n1}=\{E_{n1}\}, \ \text{or}\ S_{1n}=\{E_{n1}\} \ \text{and}\ S_{n1}=\{E_{1n}\}, \ \text{or}$ \\
			$S_{1n}=\{E_{11}\} \ \text{and}\ S_{n1}=\{E_{nn}\}, \ \text{or}\ S_{1n}=\{E_{nn}\} \ \text{and}\ S_{n1}=\{E_{11}\}$.
		\end{itemize}
	\end{prop}
	\begin{proof} First, we prove (i). From Proposition~\ref{SR_square_S11SnnS1nSn1_elements}, we have that $S_{11}$ equals one of the sets $\{E_{11}\},\{E_{nn}\}$, $\{E_{n1}\}, \{E_{1n}\}$, and similarly, $S_{nn}$ equals one of $\{E_{11}\}, \{E_{nn}\}, \{E_{n1}\}, \{E_{1n}\}$. Now, out of sixteen possible combinations of $S_{11}$ and $S_{nn}$, the four cases wherein $S_{11}=S_{nn}$ are discarded straightaway because of the ``monomiality" of $L$. It remains to show that eight other combinations of $S_{11}$ and $S_{nn}$ mentioned below are not possible:
		\begin{align*}
			&S_{11}=\{E_{11}\} ~\text{and}~ S_{nn}=\{E_{n1}\}, \quad 
			S_{11}=\{E_{11}\} ~\text{and}~ S_{nn}=\{E_{1n}\}, \\
			&S_{11}=\{E_{nn}\} ~\text{and}~ S_{nn}=\{E_{n1}\}, \quad 
			S_{11}=\{E_{nn}\} ~\text{and}~ S_{nn}=\{E_{1n}\}, \\
			&S_{11}=\{E_{n1}\} ~\text{and}~ S_{nn}=\{E_{11}\}, \quad
			S_{11}=\{E_{n1}\} ~\text{and}~ S_{nn}=\{E_{nn}\}, \\
			&S_{11}=\{E_{1n}\} ~\text{and}~ S_{nn}=\{E_{11}\}, \quad
			S_{11}=\{E_{1n}\} ~\text{and}~ S_{nn}=\{E_{nn}\}.
		\end{align*}
	Indeed, assume that $S_{11}=\{E_{11}\}$ and $S_{nn}=\{E_{n1}\}$. Let $J(c)$ be the $n\times n$ matrix obtained by multiplying the $(1,1)$ and $(n,n)$ entries of $J=J_{n\times n}$ by $c$. Then $J(c)\in \mathcal{SR}_2$ for $c>0$ and hence $\mathcal{L}(J(c)):=Y(c)\in \mathcal{SR}_2$. Now, consider the following $2 \times 2$ minors of $Y(c)$:
		\[\det Y(c) \begin{pmatrix}1,2\\1,2
		\end{pmatrix}=cy_{11}y_{22}-y_{12}y_{21},
		\qquad
		\det Y(c)\begin{pmatrix}2,n\\1,2
		\end{pmatrix}=y_{21}y_{n2}-cy_{22}y_{n1}.\]
	We can always choose $c>0$ such that the above two minors are of opposite signs, which is a contradiction. Similarly, for the remaining seven cases, we can always find two non-zero $2\times2$ minors of $Y(c)$ having opposite signs.
		
	Adapting the same argument as in the preceding half of this proof shows that (ii) holds.
	\end{proof}
	
	\begin{rem}
		In Proposition~\ref{SR_square_S11Snn_S1nSn1_combo}(i), we can assume without loss of generality that
		\begin{equation} \label{SR_square_S11Snn_wlog}
			S_{11}=\{E_{11}\}~\text{and}~S_{nn}=\{E_{nn}\}
		\end{equation}
		since $\mathcal{L}\in P(\mathcal{SR}_2)$ if and only if $\mathcal{L}$ composed with the maps $A\mapsto AP_n$ and $A\mapsto P_nA$ is a linear $\mathcal{SR}_2$-preserver. Now, using Proposition~\ref{SR_square_monomial} and equation~\eqref{SR_square_S11Snn_wlog} in Proposition~\ref{SR_square_S11Snn_S1nSn1_combo}(ii), either
		\begin{align*}
			&S_{11}=\{E_{11}\},~ S_{nn}=\{E_{nn}\},~ S_{1n}=\{E_{1n}\}, ~\text{and}~ S_{n1}=\{E_{n1}\},~ \text{or} \\
			&S_{11}=\{E_{11}\},~ S_{nn}=\{E_{nn}\},~ S_{1n}=\{E_{n1}\}, ~\text{and}~ S_{n1}=\{E_{1n}\}.
		\end{align*}
		We henceforth assume without loss of generality that
		\begin{equation}\label{SR_square_S11SnnS1nSn1_wlog}
			S_{11}=\{E_{11}\},~ S_{nn}=\{E_{nn}\},~ S_{1n}=\{E_{1n}\}, ~\text{and}~ S_{n1}=\{E_{n1}\}
		\end{equation}
		since $\mathcal{L}\in P(\mathcal{SR}_2)$ if and only if $\mathcal{L}$ composed with the map $A\mapsto A^T$ is a linear $\mathcal{SR}_2$-preserver.
	\end{rem}
	
	\begin{prop} \label{SR_square_firstLast_rowTOrow_colTOcol}
		Let $\mathcal{L}\in P(\mathcal{SR}_2)$. Then
		\begin{itemize}
			\item[(i)] $\mathcal{L}$ maps the first (and last) row and column of its arguments entirely to the first (and last) row and column, respectively.
			\item[(ii)] $\mathcal{L}$ maps each row and column of its arguments entirely to some row and column, respectively.
		\end{itemize}	
	\end{prop}
	
	\begin{proof}
		We will first show that $\mathcal{L}$ must map the first row of its arguments entirely to the first row. Again, let $J(c)$ be the matrix of size $n\times n$ obtained by multiplying the first row of $J=J_{n\times n}$ by $c>0$. Then $J(c)\in \mathcal{SR}_2$ with $\epsilon_1=1$ and hence $\mathcal{L}(J(c)):=Y(c)\in\mathcal{SR}_2$. Assume that $\mathcal{L}$ does not map the first row of its arguments entirely to the first row. Thus, there exists $1<k<n$ such that the $(1,k)$ position of the matrix $Y(c)$ is not occupied by the image of any element from the first row of $J(c)$. Using~\eqref{SR_square_S11SnnS1nSn1_wlog}, we obtain two $2 \times2$ minors of $Y(c)$:
		\[ \det Y(c)\begin{pmatrix}1,n\\1,k
		\end{pmatrix}=c\alpha
		y_{11}y_{nk}-y_{1k}y_{n1}, \qquad
		\det Y(c)\begin{pmatrix}1,n\\k,n
		\end{pmatrix}=y_{1k}y_{nn}-c\alpha
		y_{1n}y_{nk},\]
		where  either $\alpha=c$ or $\alpha=1$. Choosing $c$ large enough gives the above two minors of $Y(c)$ are of opposite signs, which is a contradiction. Similarly, the remaining assertions can be established by multiplying that particular row and column of $J=J_{n\times n}$ by $c>0$. This proves (i).
		
		To prove (ii), we begin by showing $\mathcal{L}$ must map the $k$th row of its arguments entirely to some row for $1<k<n$. Again, let $J(c)$ be the $n\times n$ matrix  obtained by multiplying the $k$th row of $J=J_{n\times n}$ by $c>0$. Then $J(c)\in \mathcal{SR}_2$ and hence $\mathcal{L}(J(c)):=Y(c)\in\mathcal{SR}_2$. By Proposition~\ref{SR_square_firstLast_rowTOrow_colTOcol}(i), $\mathcal{L}$ must map the first column of its arguments entirely to the first column and hence the $(k,1)$ element of $J(c)$ will be mapped to the first column. Let us say $\mathcal{L}$ maps it to the $(s,1)$ position of $Y(c)$, where $1<s<n$. Suppose that $\mathcal{L}$ does not map the $k$th row of $J(c)$ entirely to some row. Then there exists $1<j\leq n$ such that the $(s,j)$ position of the matrix $Y(c)$ is not occupied by the image of any element from the $k$th row of $J(c)$. By Proposition~\ref{SR_square_firstLast_rowTOrow_colTOcol}(i) and for sufficiently large $c$, we obtain a negative and a positive $2 \times 2$ minor of $Y(c)$:
		\[ \det Y(c)\begin{pmatrix}1,s\\1,j\end{pmatrix} = \det
		\begin{pmatrix}
			y_{11} & y_{1j} \\
			cy_{s1} & y_{sj}
		\end{pmatrix} < 0, \qquad
		\det Y(c)\begin{pmatrix}s,n\\1,j\end{pmatrix} = \det
		\begin{pmatrix}
			cy_{s1} & y_{sj} \\
			y_{n1} & y_{nj}
		\end{pmatrix}>0.\]
		Hence, $\mathcal{L}$ must map each row of its argument entirely to some row. Similarly, one can show that $\mathcal{L}$ must map each column of its argument entirely to some column by multiplying the $k$th column of the matrix $J=J_{n\times n}$ by $c>0$ where $1<k<n$.
	\end{proof}
	
	To summarize, we have used the transformations $A\mapsto-A$, $A\mapsto P_nA$, $A\mapsto AP_n$, and $A\mapsto A^T$ to assume that $\mathcal{L}\in P(\mathcal{SR}_2)$ has the following properties:
	\begin{itemize}
		\item[(i)] $\epsilon_1(\mathcal{L}(E_{ij}))=1$ for all $1\leq i,j\leq n$.
		
		\item[(ii)] $S_{11}=\{E_{11}\},~ S_{nn}=\{E_{nn}\},~ S_{1n}=\{E_{1n}\}, ~\text{and}~ S_{n1}=\{E_{n1}\}$.
		
		\item[(iii)] $\mathcal{L}$ maps the entire first (and last) row and column of its arguments to the first (and last) row and column, respectively.
		
		\item[(iv)] $\mathcal{L}$ maps every other row and column to some row and column, respectively.
	\end{itemize} 
	With the above information in hand, we are now ready to use induction on $n$. We first prove the base case, when $n=3$. Let $\mathcal{L}:\mathbb{R}^{3\times3}\to\mathbb{R}^{3\times3}$ be a linear map such that $\mathcal{L}(\mathcal{SR}_2)=\mathcal{SR}_2$.
	
	Let $\mathfrak{B}$ be as in~\eqref{basis}, where $m=n=3$. We have $S_{ij}=\{E_{ij}\}$ for all $1\leq i,j\leq 3$ using Propositions~\ref{SR_square_monomial}~and~\ref{SR_square_firstLast_rowTOrow_colTOcol}. Thus, $\mathcal{L}(E_{11})=l_{1}E_{11},\mathcal{L}(E_{22})=l_{2}E_{22},\ldots,\mathcal{L}(E_{32})=l_{9}E_{32}\in \mathcal{SR}_2$. Note that $l_{i}>0$ for all $1\leq i\leq 9$ by Remark~\ref{Ato-Awlog} and Proposition~\ref{SR_square_monomial}. 
	Since the rank one matrix $J = J_{3 \times 3}
	\in \mathcal{SR}_2$, hence
	\[\mathcal{L}(J)=Y=\begin{pmatrix}
		l_{1} & l_{4} & l_{5} \\
		l_{6} & l_{2} & l_{7} \\
		l_{8} & l_{9} & l_{3} 
	\end{pmatrix}\in \mathcal{SR}_2.\]
	We claim that $\mathcal{L}(J)$ has rank one. If not, it has at least one non-zero $2\times2$ minor, and all such minors are of the same sign, say non-negative without loss of generality. Now let
	\[J(c):=\begin{pmatrix}
		1 & 1 & c \\
		1 & 1 & 1 \\
		1 & 1 & 1
	\end{pmatrix}\in \mathcal{SR}_2\implies \mathcal{L}(J(c))=\begin{pmatrix}
		l_{1} & l_{4} & cl_{5} \\
		l_{6} & l_{2} & l_{7} \\
		l_{8} & l_{9} & l_{3} 
	\end{pmatrix}\in \mathcal{SR}_2.\]
	Choose $c>1$ sufficiently large such that $l_{4}l_{7}-cl_{5}l_{2}<0$. Thus, the first two columns of $\mathcal{L}(J(c))$ and hence of $\mathcal{L}(J)$ are linearly dependent. A similar consideration using $\mathcal{L}(J(c')^T)$ for $c'>1$ shows that $l_{6}l_{9}-c^{\prime}l_{2}l_{8}<0$ for $c^{\prime}$ large enough. Thus, the last two columns of $\mathcal{L}(J)$ are linearly dependent. Hence, $\mathcal{L}(J)$ is a rank 1 matrix. Now, let
	\[B=\begin{pmatrix}
		b_{11} & b_{12} & b_{13} \\
		b_{21} & b_{22} & b_{23} \\
		b_{31} & b_{32} & b_{33}
	\end{pmatrix}\in \mathcal{SR}_2 \implies \mathcal{L}(B)=\begin{pmatrix}
		l_{1}b_{11} & l_{4}b_{12} & l_{5}b_{13} \\
		l_{6}b_{21} & l_{2}b_{22} & l_{7}b_{23} \\
		l_{8}b_{31} & l_{9}b_{32} & l_{3}b_{33}
	\end{pmatrix}\in \mathcal{SR}_2.\] 
	Since $\mathcal{L}(J)$ has rank 1, we get that $\mathcal{L}$ maps $B \mapsto FBE$ (which is a positive diagonal equivalence), where $F=\begin{pmatrix}
		l_{1} & 0 & 0 \\
		0 & l_{6} & 0 \\
		0 & 0 & l_{8}
	\end{pmatrix}$ and $E=\begin{pmatrix}
		1 & 0 & 0 \\
		0 & l_{4}/l_{1} & 0 \\
		0 & 0 & l_{7}/l_{6}
	\end{pmatrix}$. This completes the base case.
	
	For the induction step, let $\mathcal{L}:\mathbb{R}^{n\times n}\to\mathbb{R}^{n\times n}$ with $n>3$ be a linear $\mathcal{SR}_2$-preserver. Let $A\in \mathcal{SR}_2$. By Proposition~\ref{SR_square_firstLast_rowTOrow_colTOcol}(i), the leading principal $(n-1)\times (n-1)$ submatrix of $A$ transforms into the leading principal $(n-1)\times (n-1)$  submatrix of $\mathcal{L}(A)$. Via padding SR$_2$ matrices by zeros $\widehat{A} \mapsto\widehat{A} \oplus (0)_{1 \times 1}$, we get that the natural restriction of $\mathcal{L}$ to the $(n-1)\times(n-1)$  leading principal submatrix is a linear $\mathcal{SR}_2$-preserver on $\mathbb{R}^{(n-1)\times(n-1)}$. By the induction hypothesis, the restriction of $\mathcal{L}$ is a composition of one or more of the following maps: (i)~$X\mapsto -X$, (ii)~$X\mapsto P_{n-1}X$, (iii)~$X\mapsto X^T$, and (iv)~$X\mapsto FXE$. In fact, it is a positive diagonal equivalence since the first row is transformed to the first row and $\epsilon_1(\mathcal{L}(E_{ij}))=1$ for all $1\leq i,j\leq n$. Thus $S_{ij}=\{E_{ij}\}$ for all $1\leq i,j\leq n-1$. By Proposition~\ref{SR_square_firstLast_rowTOrow_colTOcol}, $\mathcal{L}$ maps each row and column of its arguments entirely to some row and column, respectively, and hence 
	\begin{equation}\label{SR_sqaure_Sij=Eij}
		S_{ij}=\{E_{ij}\}~\text{for all}~ 1\leq i,j\leq n.
	\end{equation} 
	Since we may compose the inverse positive diagonal equivalence relative to the leading principal submatrix of size $(n-1)\times(n-1)$ with $\mathcal{L}$, we may assume without loss of generality that $A$ and $\mathcal{L}(A)$ have the same leading principal $(n-1) \times (n-1)$ submatrix. By~\eqref{SR_sqaure_Sij=Eij}, we have
	\[\mathcal{L}(E_{in})=c_iE_{in},~ \mathcal{L}(E_{ni})=k_iE_{ni} ~\text{for}~ 1\leq i\leq n-1, ~\text{and}~ \mathcal{L}(E_{nn})=dE_{nn}\]
	for some positive scalars $c_i$, $k_i$, and $d$. We next claim that
	$c_1=\cdots=c_{n-1}$, $k_1=\cdots=k_{n-1}$, and $d=c_1k_1$. This follows by applying $\mathcal{L}$ to the rank one matrix $J = J_{n \times n} \in \mathcal{SR}_2$ and then proceeding similarly to the $n=3$ case to get that $\mathcal{L}(J) \in \mathcal{SR}_2$ also has rank one. Hence as above, $\mathcal{L}$ maps $A$ to $FAE$ for some positive diagonal matrices $F$ and $E$. This concludes the induction step.
\end{proof}

We now show Theorem~\ref{D} for $m>n\geq2$ using the propositions at the start of this section.

\begin{proof}[Proof of Theorem~\ref{D}~(2)$\implies$(3) for $m\neq n$] 
	Let $m>n\geq2$ and $\mathcal{L}:\mathbb{R}^{m\times n}\to\mathbb{R}^{m\times n}$ be a linear map sending $\mathcal{SR}_2$ onto itself. We claim that $\mathcal{L}$ is a composition of one or more of the transformations listed in Theorem~\ref{D}(3). For convenience, again the proof is split into several propositions. The statements of these propositions are similar to the case $m=n\geq3$, but the proofs differ when $m>n\geq2$.
	
 \begin{prop}\label{SR_rectangle_S11SmnS1nSm1_elements}
		For $\mathcal{L}\in P(\mathcal{SR}_2)$, the element in each of the sets $S_{11}$, $S_{mn}$, $S_{1n}$, and $S_{m1}$ must be one of $E_{11}$, $E_{mn}$, $E_{1n}$, $E_{m1}$.
	\end{prop}
	
	\begin{proof}
		We prove this by contradiction. Suppose $S_{11}\neq\{E_{11}\}, \{E_{mn}\},	\{E_{1n}\}, \{E_{m1}\}$. We first consider the case $n=2$. Let
		\[J(c):=\begin{pmatrix}
			c & 1 \\
			1 & 1 \\
			\vdots & \vdots \\
			1 & 1
		\end{pmatrix}_{m\times2}\in\mathcal{SR}_2\implies\text{either}~ \mathcal{L}(J(c))=\begin{pmatrix}
			y_{11} & y_{12} \\
			\vdots & \vdots \\
			cy_{i1} & y_{i2} \\
			\vdots & \vdots \\
			y_{m1} & y_{m2}
		\end{pmatrix}~\text{or}~ \mathcal{L}(J(c))=\begin{pmatrix}
			y_{11} & y_{12} \\
			\vdots & \vdots \\
			y_{j1} & cy_{j2} \\
			\vdots & \vdots \\
			y_{m1} & y_{m2}
		\end{pmatrix}.\]
		Since $m>2$, we can find two $2\times2$ minors of $\mathcal{L}(J(c))$ which are of opposite signs for large $c$, a contradiction. Now, for $m>n\geq3$, the same arguments used in Proposition~\ref{SR_square_S11SnnS1nSn1_elements} can be used here.
		
		Similarly, we can show this for each of $S_{mn}$, $S_{1n}$, and $S_{m1}$ by multiplying the $(m,n)$, $(1,n)$, and $(m,1)$ entries, respectively, of $J=J_{m\times n}$ by $c>0$.
	\end{proof}
	\begin{prop}\label{Proposition_L(J)_rank_1}
		Let $m>n\geq2$. For $\mathcal{L}\in P(\mathcal{SR}_2)$, the matrix $\mathcal{L}(J_{m\times n})$ has rank one.
	\end{prop}
	\begin{proof}
		Let $J = J_{m\times n}$; we claim that the	matrix
		$\mathcal{L}(J)=\begin{pmatrix}
			y_{11} & \cdots & y_{1n} \\
			\vdots & \ddots & \vdots \\
			y_{m1} & \cdots & y_{mn}
		\end{pmatrix}$ has rank one. Indeed, if $\mathcal{L}(J)$ has rank 2, then it has at least one non-zero minor of size $2\times2$ and moreover such minors have the same sign, say non-negative without loss of generality. By Proposition~\ref{SR_rectangle_S11SmnS1nSm1_elements}, each of the sets $S_{11}$, $S_{mn}$, $S_{1n}$, and $S_{m1}$ must contain exactly one of the standard basis elements $E_{11}$, $E_{mn}$, $E_{1n}$, $E_{m1}$. Again without loss of generality assume that $S_{11} = \{E_{11}\}$. Now, consider the matrix
		\[J(c):=\begin{pmatrix}
			c & \cdots & 1 \\
			\vdots & \ddots & \vdots \\
			1 & \cdots & 1 
		\end{pmatrix}\in \mathcal{SR}_2 \implies \mathcal{L}(J(c))=\begin{pmatrix}
			cy_{11} & \cdots & y_{1n} \\
			\vdots & \ddots & \vdots  \\
			y_{m1} & \cdots & y_{mn}  
		\end{pmatrix}\in \mathcal{SR}_2.\] 
		But we can choose $c>0$ sufficiently small such that $cy_{11}y_{22}-y_{12}y_{21}<0$. Thus, the span of the last $m-1$ rows of $\mathcal{L}(J(c))$ has dimension one, and hence all $2\times2$ minors of $\mathcal{L}(J)$ in the last $m-1$ rows are zero.
		
		Now, among the sets $S_{1n}, S_{m1}$, and $S_{mn}$, exactly one of them must contain $E_{mn}$; suppose this is $S_{pq}$. By multiplying the $(p,q)$ element of $J$ by $c$ and proceeding as above, one can conclude that the span of the first $m-1$ rows of $\mathcal{L}(J)$ has dimension one. Thus, $\mathcal{L}(J)$ has rank one.
	\end{proof}
	
	\begin{prop}\label{SR_rectangle_S11Smn_S1nSm1_combo} 
		For $\mathcal{L}\in P(\mathcal{SR}_2)$, the following pairwise combinations are possible.
		\begin{itemize}
			\item[(i)] $S_{11}=\{E_{11}\} ~\text{and}~ S_{mn}=\{E_{mn}\}, ~\text{or}~ S_{11}=\{E_{mn}\} ~\text{and}~ S_{mn}=\{E_{11}\}, ~\text{or}$ \\
			$S_{11}=\{E_{1n}\} ~\text{and}~ S_{mn}=\{E_{m1}\}, ~\text{or}~ S_{11}=\{E_{m1}\} ~\text{and}~ S_{mn}=\{E_{1n}\}$.
			
			\item[(ii)] $S_{1n}=\{E_{1n}\} ~\text{and}~ S_{m1}=\{E_{m1}\},~\text{or}~ S_{1n}=\{E_{m1}\} ~\text{and}~ S_{m1}=\{E_{1n}\}, ~\text{or}$ \\
			$S_{1n}=\{E_{11}\} ~\text{and}~ S_{m1}=\{E_{mn}\},~\text{or}~ S_{1n}=\{E_{mn}\} ~\text{and}~ S_{m1}=\{E_{11}\}$.
		\end{itemize}
	\end{prop}

	\begin{proof} First, we prove (i). The proof is similar to that
	of Proposition~\ref{SR_square_S11Snn_S1nSn1_combo}, except the $n=2$ case. From Proposition~\ref{SR_rectangle_S11SmnS1nSm1_elements}, we have that $S_{11}$ equals one of the sets $\{E_{11}\}, \{E_{mn}\}, \{E_{m1}\}, \{E_{1n}\}$ and similarly, $S_{mn}$ is one of $\{E_{11}\}, \{E_{mn}\}, \{E_{m1}\}, \{E_{1n}\}$. Since $L$ is a  monomial matrix, out of sixteen possible combinations of $S_{11}$ and $S_{mn}$, the four cases wherein $S_{11}=S_{mn}$ are discarded straightaway. We next show that the following eight other combinations of $S_{11}$ and $S_{mn}$ are also not possible:
	\begin{align*}
		&S_{11}=\{E_{m1}\} ~\text{and}~ S_{mn}=\{E_{mn}\}, \quad 
		S_{11}=\{E_{11}\} ~\text{and}~ S_{mn}=\{E_{1n}\}, \\
		&S_{11}=\{E_{mn}\} ~\text{and}~ S_{mn}=\{E_{m1}\}, \quad 
		S_{11}=\{E_{mn}\} ~\text{and}~ S_{mn}=\{E_{1n}\}, \\
		&S_{11}=\{E_{m1}\} ~\text{and}~ S_{mn}=\{E_{11}\}, \quad
		S_{11}=\{E_{11}\} ~\text{and}~ S_{mn}=\{E_{m1}\}, \\
		&S_{11}=\{E_{1n}\}~\text{and}~ S_{mn}=\{E_{11}\}, \quad
		S_{11}=\{E_{1n}\}~\text{and}~ S_{mn}=\{E_{mn}\}.
	\end{align*}
	Suppose $S_{11}=\{E_{m1}\}$ and $S_{mn}=\{E_{mn}\}$. Let $n=2$ and for $c>0$, define
		\[J(c):= \begin{pmatrix}
			1+c \qquad c \\
			J_{(m-1) \times 2}
		\end{pmatrix}\in \mathcal{SR}\implies \mathcal{L}(J(c))\in \mathcal{SR}.\]
	By Propositions~\ref{SR_square_monomial}~and~\ref{SR_rectangle_S11SmnS1nSm1_elements}, $\mathcal{L}(J(c))$ can be either of the following:
	\begin{align*}
		\begin{pmatrix}
			y_{11} & cy_{12} \\
			y_{21} & y_{22} \\
			\vdots & \vdots \\
			(1+c)y_{m1} & y_{m2}
		\end{pmatrix}~\text{or}~\begin{pmatrix}
			cy_{11} & y_{12} \\
			y_{21} & y_{22} \\
			\vdots & \vdots \\
			(1+c)y_{m1} & y_{m2}
		\end{pmatrix}.
	\end{align*}
	For the first case, using the fact that $\mathcal{L}(J_{m\times2})$ has rank 1 by Proposition~\ref{Proposition_L(J)_rank_1}, the following two $2\times2$ minors of $Y(c)$ have opposite signs for an appropriate choice of $c$:
		\[\det\begin{pmatrix}
			y_{11} & cy_{12} \\
			y_{21} & y_{22} 
		\end{pmatrix}>0~\text{and}~
		\det\begin{pmatrix}
			y_{21} & y_{22} \\
			(1+c)y_{m1} & y_{m2} 
		\end{pmatrix}<0.\]
	For the other case, we can choose $c$ large enough such that we obtain two minors of size $2\times2$ of $Y(c)$ having opposite signs. Hence, we conclude that $S_{11}\neq\{E_{m1}\}$ and $S_{mn}\neq\{E_{mn}\}$ for $n=2$. Now the proof for $m>n\geq3$ follows similarly to the proof of Proposition~\ref{SR_square_S11Snn_S1nSn1_combo}.
	
	Using a similar argument as in the above proof shows that (ii) holds.
	\end{proof}

	\begin{rem}
		In Proposition~\ref{SR_rectangle_S11Smn_S1nSm1_combo}(i), we can assume without loss of generality that
		\begin{equation}\label{SR_rectangle_S11Smn_wlog}
			S_{11}=\{E_{11}\}~\text{and}~S_{mn}=\{E_{mn}\}
		\end{equation}
		since $\mathcal{L}\in P(\mathcal{SR})$ if and only if $\mathcal{L}$ composed with the maps $A\mapsto AP_n$ and $A\mapsto P_mA$ is a linear $\mathcal{SR}$-preserver. Now, using Proposition~\ref{SR_square_monomial} and equation~\eqref{SR_rectangle_S11Smn_wlog} in  Proposition~\ref{SR_rectangle_S11Smn_S1nSm1_combo}(ii), we have either
		\begin{align*}
			&S_{11}=\{E_{11}\},~ S_{mn}=\{E_{mn}\},~ S_{1n}=\{E_{1n}\}, ~\text{and}~ S_{m1}=\{E_{m1}\},~\text{or} \\
			&S_{11}=\{E_{11}\},~ S_{mn}=\{E_{mn}\},~ S_{1n}=\{E_{m1}\}, ~\text{and}~ S_{m1}=\{E_{1n}\}.
		\end{align*}
	\end{rem}
	Next, we show that the conditions $S_{11}=\{E_{11}\}$, $S_{mn}=\{E_{mn}\}$, $S_{1n}=\{E_{m1}\}$, $S_{m1}=\{E_{1n}\}$ do not simultaneously hold. Indeed, say they all hold; let $J(c)$ be the matrix obtained by multiplying the first row of $J=J_{m\times n}$ by $c>0$. Then $J(c)\in \mathcal{SR}_2$ and so $\mathcal{L}(J(c)):=Y(c)\in \mathcal{SR}_2$. Since $m>n$, hence even if $\mathcal{L}$ maps entire elements of the first row of $J(c)$ to the first column of $Y(c)$, there exists $1<j<m$ such that the $(j,1)$ position of $Y(c)$ is not occupied by the image of any element from the first row of $J(c)$. Consider the following $2 \times 2$ minors of $Y(c)$:
	\[ \det Y(c)\begin{pmatrix}1,j\\1,n\end{pmatrix} =
	\det\begin{pmatrix}
		cy_{11} & y_{1n} \\
		y_{j1} & \alpha y_{jn}
	\end{pmatrix}, \qquad
	\det Y(c)\begin{pmatrix}j,m\\1,n\end{pmatrix}
	= \det\begin{pmatrix}
		y_{j1} & \alpha y_{jn} \\
		cy_{m1} & y_{mn}
	\end{pmatrix},\]
	where either $\alpha=c$ or $\alpha=1$. It is always possible to choose $c$ large enough such that the above two minors of $Y(c)$ are of opposite signs, which is a contradiction. Thus, we have
	\begin{equation}\label{SR_rectangle_S11SmnS1nSm1_wlog}
		S_{11}=\{E_{11}\},~ S_{mn}=\{E_{mn}\},~ S_{1n}=\{E_{1n}\}, ~\text{and}~ S_{m1}=\{E_{m1}\}. 
	\end{equation}
	
	\begin{prop}\label{SR_rectangle_FirstLast_rowTorow_colTOcol}
		Let $\mathcal{L}\in P(\mathcal{SR}_2)$ such that \eqref{SR_rectangle_S11SmnS1nSm1_wlog} holds.
		\begin{itemize}
			\item[(i)] Then $\mathcal{L}$ must map the first (and last) row and column of its arguments entirely to the first (and last) row and column, respectively.
			
			\item[(ii)] Moreover, $\mathcal{L}$ must map each row and column of its arguments entirely to some row and column, respectively.
		\end{itemize}
	\end{prop}	
	\begin{proof}
		We begin by showing (i). We first show that $\mathcal{L}$ must map the first (and last) row of its arguments entirely to the first (and last) row. For $n=2$, this follows directly from~\eqref{SR_rectangle_S11SmnS1nSm1_wlog}. For $m>n\geq3$, the proof is similar to that of Proposition~\ref{SR_square_firstLast_rowTOrow_colTOcol}.
		
		Our next aim is to prove that $\mathcal{L}$ maps the entire first column of its arguments to the first column. Let $J(c)$ be the matrix obtained by multiplying the first column of $J=J_{m\times n}$ by $c>0$. Assume to the contrary that there exists $k$ where $1<k<m$ such that the $(k,1)$ position of the matrix $Y(c)$ is not occupied by the image of any element from the first column of $J(c)$. Since $\mathcal{L}$ maps the entire first and last row of its arguments to the first and the last row, respectively, the following two $2 \times 2$ minors of $Y(c)$ give us a contradiction, with either $\alpha = c$ or $\alpha = 1$:
		\[ \det Y(c)\begin{pmatrix}1,k\\1,n\end{pmatrix}=
		\det\begin{pmatrix}
			cy_{11} & y_{1n} \\
			y_{k1} & \alpha y_{kn}
		\end{pmatrix}, \qquad
		\det Y(c)\begin{pmatrix}k,m\\1,n\end{pmatrix}
		= \det\begin{pmatrix}
			y_{k1} & \alpha y_{kn} \\
			cy_{m1} & y_{mn}
		\end{pmatrix}.\]
		We can prove similarly the corresponding assertion for the last column as well. This shows (i). We now show (ii). The proof is similar to Proposition~\ref{SR_square_firstLast_rowTOrow_colTOcol}(ii), except for the case when $n=2$. By Proposition~\ref{SR_rectangle_FirstLast_rowTorow_colTOcol}(i), the assertion holds trivially for columns. Next we claim that $\mathcal{L}$ maps each row of its argument to some row. Let $J(c)$ be the $m\times2$ matrix obtained by multiplying the $k$th row of $J=J_{m\times 2}$ by $c>0$ where $1<k<m$. By Proposition~\ref{SR_rectangle_FirstLast_rowTorow_colTOcol}(i), $\mathcal{L}$ maps the entire first (and last) column of its arguments to the first (and last) column and hence the $(k,1)$ element of $J(c)$ will be mapped to the first column; let us say $\mathcal{L}$ maps it to the $(p,1)$ position of $Y(c)$. To the contrary assume that $\mathcal{L}$ does not map the entire $k$th row of its argument to some row, then the $(p,2)$ position of $Y(c)$ is not occupied by the image of the element of $J(c)$ present in the $(k,2)$ position. Using Proposition~\ref{SR_rectangle_FirstLast_rowTorow_colTOcol}(i), we obtain the following $2 \times 2$ minors of $Y(c)$ having opposite signs for an appropriate choice of $c$, which completes the proof:
		\[ \det Y(c)\begin{pmatrix}1,p\\1,2\end{pmatrix} =
		\det\begin{pmatrix} y_{11} & y_{12} \\ cy_{p1} & y_{p2}
		\end{pmatrix} \ < \ 0 \ < \
		\det Y(c)\begin{pmatrix}p,m\\1,2\end{pmatrix} =
		\det\begin{pmatrix} cy_{p1} & y_{p2} \\ y_{m1} & y_{m2}
		\end{pmatrix}. \qedhere
		\]
	\end{proof}
	So far, we have used the transformations $A\mapsto-A$, $A\mapsto P_mA$, and $A\mapsto AP_n$ to assume that $\mathcal{L}\in P(\mathcal{SR}_2)$ has the following properties:
	\begin{itemize}
		\item[(i)] $\epsilon_1(\mathcal{L}(E_{ij}))=1$ for all $1\leq i\leq m$, $1\leq j\leq n$.
		
		\item[(ii)] $S_{11}=\{E_{11}\},~ S_{mn}=\{E_{mn}\},~ S_{1n}=\{E_{1n}\}, ~\text{and}~ S_{m1}=\{E_{m1}\}$.
		
		\item[(iii)]  $\mathcal{L}$ maps the entire first (and last) row and column of its arguments to the first (and last) row and column, respectively.
		
		\item[(iv)] $\mathcal{L}$ maps every other row and column to some row and column, respectively.
	\end{itemize} 
	With the above analysis in hand, we complete the proof by induction on the sizes of the matrices. We first show by induction on $m>2$ that the result holds for all $m\times2$ matrices. For the base case, let $m=3$ and $\mathcal{L}:\mathbb{R}^{3\times2}\to\mathbb{R}^{3\times2}$ be a linear map such that $\mathcal{L}(\mathcal{SR}_2)=\mathcal{SR}_2$. 
	
	Let $\mathfrak{B}=(E_{11}, E_{22}, E_{12}, E_{21}, E_{31}, E_{32})$ be an ordered basis of $\mathbb{R}^{3\times 2}$. We have $S_{ij}=\{E_{ij}\}$ for all $1\leq i\leq 3, ~ 1\leq j\leq 2$ by Proposition~\ref{SR_rectangle_FirstLast_rowTorow_colTOcol}. Thus, $\mathcal{L}(E_{11})= l_{1}E_{11},\ldots,\mathcal{L}(E_{32})= l_{6}E_{32}\in \mathcal{SR}_2$, where $l_{i}>0$ for all $1\leq i\leq 6$ by Remark~\ref{Ato-Awlog} and Proposition~\ref{SR_square_monomial}. 
	From Proposition~\ref{Proposition_L(J)_rank_1},
	$\mathcal{L}(J_{3\times2})=\begin{pmatrix}
		l_{1} & l_{3}  \\
		l_{4} & l_{2} \\
		l_{5} & l_{6}  
	\end{pmatrix}$ has rank 1. Let $B=\begin{pmatrix}
		b_{11} & b_{12} \\
		b_{21} & b_{22} \\
		b_{31} & b_{32}
	\end{pmatrix}\in \mathcal{SR}_2.$ Then $\mathcal{L}(B)=\begin{pmatrix}
		l_{1}b_{11} & l_{3}b_{12} \\
		l_{4}b_{21} & l_{2}b_{22} \\
		l_{5}b_{31} & l_{6}b_{32}
	\end{pmatrix}\in \mathcal{SR}_2.$ Since $\mathcal{L}(J)$ has rank 1,
	\[\mathcal{L}(B)=\begin{pmatrix}
		l_{1} & 0 & 0 \\
		0 & l_{4} & 0 \\
		0 & 0 & l_{5}
	\end{pmatrix}\begin{pmatrix}
		b_{11} & b_{12} \\
		b_{21} & b_{22} \\
		b_{31} & b_{32}
	\end{pmatrix}\begin{pmatrix}
		1 & 0 \\
		0 & l_{3}/l_{1}
	\end{pmatrix},\] 
	which is a positive diagonal equivalence. This completes the base case. \smallskip
	
	Strategy used for applying induction: For the induction step, we first prove it for the vector space of $m\times2$ matrices assuming it to be true for that of $(m-1)\times2$ matrices, where $m>3$. After this, we again apply induction on $m\times n$ by assuming it holds for the space of $m\times (n-1)$ matrices for fixed $m$, where $m>n$. 
	
	Let $A\in\mathcal{SR}_2$ have size $m\times2$. By Proposition~\ref{SR_rectangle_FirstLast_rowTorow_colTOcol}(i), the submatrix of $A$ formed by the first $(m-1)$ rows and both columns must be transformed to the first $(m-1)$ rows of $\mathcal{L}(A)$. Since every SR$_2$ matrix $\widehat{A}$ of size $(m-1)\times 2$ is a submatrix of the SR$_2$ matrix $(\widehat{A}^T|\mathbf{0})^T\in\mathbb{R}^{m\times 2}$, the natural restriction of $\mathcal{L}$ onto the $(m-1)\times 2$ leading submatrix is a linear $\mathcal{SR}_2$-preserver on $\mathbb{R}^{(m-1)\times 2}$. By the induction hypothesis, the restriction of $\mathcal{L}$ is a composition of one or more of the following maps: (i)~$X\mapsto -X$, (ii)~$X\mapsto XP_2$, (iii)~$X\mapsto P_{m-1}X$, and (iv)~$X\mapsto FXE$. In fact, it is a positive diagonal equivalence since the first row and column are transformed to the first row and column, respectively, with $\epsilon_1(\mathcal{L}(E_{ij}))=1$ for all $i,j$. Thus $S_{ij}=\{E_{ij}\}$ for all $1\leq i\leq m-1$ and $1\leq j\leq 2$.
	By Proposition~\ref{SR_rectangle_FirstLast_rowTorow_colTOcol}(i), $\mathcal{L}$ maps the first (and last) column of its arguments entirely to the first (and last) column and hence
	\begin{equation}\label{SR_rectangle_Sij=Eij}
		S_{ij}=\{E_{ij}\} ~\text{for all}~ 1\leq i\leq m,~ 1\leq j\leq 2.
	\end{equation} 
	Since we may compose the inverse positive diagonal equivalence relative to the upper submatrix of size $(m-1)\times 2$ with $\mathcal{L}$, we may assume without loss of generality that all but the last row of $A$ and $\mathcal{L}(A)$ are equal. Using~\eqref{SR_rectangle_Sij=Eij}, we have
	$\mathcal{L}(E_{mi})=k_iE_{mi}~\text{for}~ 1\leq i\leq 2,$ for some positive scalars $k_1$ and $k_2$. We next claim that $k_1=k_{2}.$ By Proposition~\ref{Proposition_L(J)_rank_1},
	$\mathcal{L}(J)=\begin{pmatrix} J_{(m-1) \times 2} \\ k_1 \quad \; k_2 \end{pmatrix}\in \mathcal{SR}_2$ has rank one, and  this implies $k_1=k_{2}$. Therefore $\mathcal{L}$ is a positive diagonal equivalence. This completes the induction step for $m$. 
	
	Next, fix arbitrary $m$ and suppose the assertion holds for linear $\mathcal{SR}_2$-preservers from $\mathbb{R}^{m\times (n-1)}$ to $\mathbb{R}^{m\times (n-1)}$, where $m>n$. Let $\mathcal{L}:\mathbb{R}^{m\times n}\to\mathbb{R}^{m\times n}$ such that $\mathcal{L}\in P(\mathcal{SR}_2)$. For $A\in \mathcal{SR}_2$ of size $m\times n$, the submatrix of $A$ formed by all rows and the first $(n-1)$ columns must be transformed to the first $(n-1)$ columns of $\mathcal{L}(A)$ because of Proposition~\ref{SR_rectangle_FirstLast_rowTorow_colTOcol}(i). Since every SR$_2$ matrix $\widehat{A}$ of size $m\times(n-1)$ is a submatrix of the SR$_2$ matrix $(\widehat{A}|\mathbf{0})\in\mathbb{R}^{m\times n}$, the natural restriction of $\mathcal{L}$ onto the $m\times(n-1)$ left submatrix is a linear $\mathcal{SR}_2$-preserver on $\mathbb{R}^{m\times(n-1)}$. By the induction hypothesis, it is a composition of one or more of the following maps: (i)~$X\mapsto -X$, (ii)~$X\mapsto XP_{n-1}$, (iii)~$X\mapsto P_mX$, and (iv)~$X\mapsto FXE$. By the same arguments in the preceding part, we have
	\[S_{ij}=\{E_{ij}\} ~\text{for all}~ 1\leq i\leq m,~ 1\leq j\leq n-1.\] 
	By Proposition~\ref{SR_rectangle_FirstLast_rowTorow_colTOcol}, $\mathcal{L}$ maps each row of its arguments entirely to some row and hence
	\begin{equation}\label{SR_rectangle_second_Sij=Eij}
		S_{ij}=\{E_{ij}\} ~\text{for all}~ 1\leq i\leq m,~ 1\leq j\leq n.
	\end{equation} 
	Since we may compose the inverse positive diagonal equivalence relative to the left submatrix of size $m\times(n-1)$ with $\mathcal{L}$, we may assume without loss of generality that all but the last column of $A$ and $\mathcal{L}(A)$ are equal. Using~\eqref{SR_rectangle_second_Sij=Eij}, we have
	\[\mathcal{L}(E_{in})=c_iE_{in}~\text{for}~ 1\leq i\leq m, ~\text{for some positive scalar} ~c_i.\]
	Now, to complete the proof, we must show that $c_1=\cdots=c_{m}$. Since the rank one matrix $J = J_{m \times n} \in \mathcal{SR}$, so $\mathcal{L}(J)\in \mathcal{SR}$. By Proposition~\ref{Proposition_L(J)_rank_1}, $\mathcal{L}(J_{m\times n})$ has rank one. Hence, all $2\times2$ minors of $\mathcal{L}(J)$ are zero, and so $c_1=\cdots=c_{m}$. Thus, $\mathcal{L}$ maps $A$ to $FAE$ for some positive diagonal matrices $F$ and $E$. This concludes the induction step.
\end{proof}
\begin{proof}[Proof of Theorem~\ref{D}~(1)$\implies$(3)]
We show this simultaneously for both cases: when $m=n$ and $m\neq n$. In the proof of (2)$\implies$(3), note that all the test matrices $A$ considered were SR, and we only dealt with the $1\times1$ and $2\times2$ minors of the output matrices $\mathcal{L}(A)$. Therefore, even if (1) holds instead of (2), the same steps listed above yield the preservers listed in (3).
\end{proof}

\begin{rem}\label{SSR_proof}
	Theorem~\ref{D} holds also for the class of SSR matrices. Since the five linear transformations specified in Theorem~\ref{D} map $\mathcal{SSR}$ onto $\mathcal{SSR}$, so $P(\mathcal{SR})\subseteq P(\mathcal{SSR})$. By Theorem~\ref{G-K}, $\mathcal{SR}=\overline{\mathcal{SSR}}$, and by Lemma~\ref{lemma_preserver}, it follows that $P(\mathcal{SSR})\subseteq P(\mathcal{SR})$. Therefore, $P(\mathcal{SSR})=P(\mathcal{SR})$.
\end{rem}

To complement Theorem~\ref{D}, we now classify all linear $\mathcal{SR}$-preservers on $\mathbb{R}^{2\times2}$.

\begin{proof}[Proof of Theorem~\ref{thrmSR_1}]
	The implications (1)$\implies$(2) and (3)$\implies$(1) are trivial. It only remains to show (2)$\implies$(3). Let $\mathcal{L}:\mathbb{R}^{2\times2}\to\mathbb{R}^{2\times2}$ be a linear map such that $\mathcal{L}(\mathcal{SR}_1)=\mathcal{SR}_1$ and let $\mathfrak{B}=(E_{11},E_{22},E_{12},E_{21})$ be an ordered basis of $\mathbb{R}^{2\times2}$. Since $E_{ij}\in\mathcal{SR}_1$ for all $1\leq i,j\leq2$ and $\mathcal{L}$ maps $\mathcal{SR}_1$ onto itself, therefore $\mathcal{L}^{-1}$ exists, and further $\mathcal{L}^{-1}\in P(\mathcal{SR}_1)$. Notice that the proof of Proposition~\ref{SR_square_image_same_epsilon_1_sign} holds for $\mathcal{L}\in P(\mathcal{SR}_1)$. Thus $\epsilon_1(\mathcal{L}(E_{ij}))$ has the same sign for all $i,j$. We can assume without loss of generality that $\epsilon_1(\mathcal{L}(E_{ij}))=1$ for all $i,j$ since $\mathcal{L}$ is a linear $\mathcal{SR}_1$-preserver if and only if $\mathcal{L}$ composed with the map $A\mapsto-A$ is a linear $\mathcal{SR}_1$-preserver. By Proposition~\ref{SR_square_monomial} and Remark~\ref{Sij_singleton}, $S_{ij}$ is a non-empty singleton set for all $1\leq i,j\leq 2$ and $S_{ij}\cap S_{kl}=\emptyset$ for $(i,j)\neq(k,l)$.  
	%
	
	Thus the element in each of the sets $S_{11},$ $S_{12},$ $S_{21}$, and $S_{22}$ must be one of the following:
	\[E_{11},~E_{12},~E_{21},~\text{and}~ E_{22}.\]
	Thus, there are 24 possible combinations of the sets $S_{11},$ $S_{12},$ $S_{21}$, and $S_{22}$ -- listed in Table~\ref{Ttable}.
	\begin{table}[H]
		\begin{tabular}{ ||c|c|c|c|c|c||} 
			\hline\hline
			S.No. & $S_{11}$ & $S_{12}$ & $S_{21}$ & $S_{22}$ & $\mathcal{SR}_1$-preservers \\ 
			\hline\hline
			I. & $E_{11}$  &  $E_{12}$  &  $E_{21}$  &  $E_{22}$  & - \\
			& $E_{11}$  &  $E_{21}$  &  $E_{12}$  &  $E_{22}$  & $A\mapsto A^T$ \\
			& $E_{12}$  &  $E_{11}$  &  $E_{22}$  &  $E_{21}$  & $A\mapsto AP_2$  \\
			& $E_{12}$  &  $E_{22}$  &  $E_{11}$  &  $E_{21}$  & $A\mapsto AP_2\mapsto (AP_2)^T$ \\
			& $E_{21}$  &  $E_{11}$  &  $E_{22}$  &  $E_{12}$  & $A\mapsto P_2A\mapsto (P_2A)^T$ \\
			& $E_{21}$  &  $E_{22}$  &  $E_{11}$  &  $E_{12}$  & $A\mapsto P_2A$ \\
			& $E_{22}$  &  $E_{21}$  &  $E_{12}$  &  $E_{11}$  & $A\mapsto P_2AP_2$ \\
			& $E_{22}$  &  $E_{12}$  &  $E_{21}$  &  $E_{11}$  & $A\mapsto P_2AP_2\mapsto (P_2AP_2)^T$ \\
			\hline 
			II. & $E_{11}$  &  $E_{12}$  &  $E_{22}$  &  $E_{21}$  & - \\
			& $E_{11}$  &  $E_{21}$  &  $E_{22}$  &  $E_{12}$  & $A\mapsto A^T$ \\
			& $E_{12}$  &  $E_{11}$  &  $E_{21}$  &  $E_{22}$  & $A\mapsto AP_2$ \\
			& $E_{12}$  &  $E_{22}$  &  $E_{21}$  &  $E_{11}$  & $A\mapsto AP_2\mapsto (AP_2)^T$ \\
			& $E_{21}$  &  $E_{11}$  &  $E_{12}$  &  $E_{22}$  & $A\mapsto P_2A\mapsto (P_2A)^T$ \\
			& $E_{21}$  &  $E_{22}$  &  $E_{12}$  &  $E_{11}$  & $A\mapsto P_2A$ \\
			& $E_{22}$  &  $E_{21}$  &  $E_{11}$  &  $E_{12}$  & $A\mapsto P_2AP_2$  \\
			& $E_{22}$  &  $E_{12}$  &  $E_{11}$  &  $E_{21}$  & $A\mapsto P_2AP_2\mapsto (P_2AP_2)^T$  \\
			\hline
			III. & $E_{11}$  &  $E_{22}$  &  $E_{12}$  &  $E_{21}$  & - \\
			& $E_{11}$  &  $E_{22}$  &  $E_{21}$  &  $E_{12}$  & $A\mapsto A^T$  \\
			& $E_{12}$  &  $E_{21}$  &  $E_{11}$  &  $E_{22}$  & $A\mapsto AP_2$  \\
			& $E_{12}$  &  $E_{21}$  &  $E_{22}$  &  $E_{11}$  & $A\mapsto AP_2\mapsto (AP_2)^T$ \\
			& $E_{21}$  &  $E_{12}$  &  $E_{11}$  &  $E_{22}$  & $A\mapsto P_2A\mapsto (P_2A)^T$  \\
			& $E_{21}$  &  $E_{12}$  &  $E_{22}$  &  $E_{11}$  & $A\mapsto P_2A$  \\
			& $E_{22}$  &  $E_{11}$  &  $E_{21}$  &  $E_{12}$  & $A\mapsto P_2AP_2$ \\
			& $E_{22}$  &  $E_{11}$  &  $E_{12}$  &  $E_{21}$  & $A\mapsto P_2AP_2\mapsto (P_2AP_2)^T$  \\
			\hline\hline
		\end{tabular}
		
	\caption{Table listing the 24 cases in the proof of	Theorem~\ref{thrmSR_1}.}\label{Ttable}
	\end{table}
	In the last column of this table, we have listed the $\mathcal{SR}_1$-preservers that we have used to assume without loss of generality:
	\begin{align}
		\mathrm{I.} & ~ S_{11}=\{E_{11}\}, ~ S_{12}=\{E_{12}\}, ~ S_{21}=\{E_{21}\}, ~ \text{and} ~ S_{22}=\{E_{22}\};\label{SR_1preserver-1}\\
		\mathrm{II.} & ~ S_{11}=\{E_{11}\}, ~ S_{12}=\{E_{12}\}, ~ S_{21}=\{E_{22}\}, ~ \text{and} ~ S_{22}=\{E_{21}\};\label{SR_1preserver-2}\\
		\mathrm{III.} & ~ S_{11}=\{E_{11}\}, ~ S_{12}=\{E_{22}\}, ~ S_{21}=\{E_{12}\}, ~ \text{and} ~ S_{22}=\{E_{21}\}.\label{SR_1preserver-3}
	\end{align}
	Further, using the map $\mathcal{M}:\begin{pmatrix}
		a_{11} & a_{12} \\
		a_{21} & a_{22}
	\end{pmatrix}\mapsto\begin{pmatrix}
		a_{11} & a_{12} \\
		a_{22} & a_{21}
	\end{pmatrix}$ in~\eqref{SR_1preserver-2} and the composition of the maps $\mathcal{M}$ and $A\mapsto A^T$ in~\eqref{SR_1preserver-3}, we assume~\eqref{SR_1preserver-1} without loss of generality, since $\mathcal{L}\in P(\mathcal{SR}_1)$ if and only if $\mathcal{L}$ composed with the map $\mathcal{M}$ is a linear $\mathcal{SR}_1$-preserver.
	
	Let $\mathcal{L}(E_{11})=l_1E_{11}$, $\mathcal{L}(E_{22})=l_2E_{22}$, $\mathcal{L}(E_{12})=l_3E_{12}$, and $\mathcal{L}(E_{21})=l_4E_{21}$, where $l_i>0$ for all $1\leq i\leq4$ as $\epsilon_1(\mathcal{L}(E_{ij}))=1$ for all $i,j$. Thus for any $2\times2$ matrix $B=\begin{pmatrix}
		b_{11} & b_{12} \\
		b_{21} & b_{22}
	\end{pmatrix}\in\mathcal{SR}_1$, $\mathcal{L}(B)=\begin{pmatrix}
		l_1b_{11} & l_3b_{12} \\
		l_4b_{21} & l_2b_{22}
	\end{pmatrix}=H\circ B$, where $H=\begin{pmatrix}
		l_1 & l_3 \\
		l_4 & l_2
	\end{pmatrix}$. Therefore, the linear $\mathcal{SR}_1$-preservers
	on $\mathbb{R}^{2\times2}$ are compositions of one or more of the
	following maps: (a)~$A\mapsto H\circ A$, where $H\in (0,\infty)^{2\times 2}$, (b)~$A\mapsto-A$, (c)~$A\mapsto P_2A$, (d)~$A\mapsto A^T$, and (e)~$\begin{pmatrix}
		a_{11} & a_{12} \\
		a_{21} & a_{22}
	\end{pmatrix}\mapsto\begin{pmatrix}
		a_{11} & a_{12} \\
		a_{22} & a_{21}
	\end{pmatrix}$.
	
	The result for SSR$_1$ follows similarly to the argument in Remark~\ref{SSR_proof}. This completes the proof.
\end{proof}

As a consequence of Theorem~\ref{thrmSR_1}, we next classify the linear preservers of $2\times2$ SSR matrices.

\begin{theorem}\label{Theorem_SSR2}
	Given a linear map $\mathcal{L}:\mathbb{R}^{2\times2}\to\mathbb{R}^{2\times 2}$, the following are equivalent.
	\begin{itemize}
		\item[(1)] $\mathcal{L}$ maps the class of $2\times 2$ SSR matrices onto itself.
		\item[(2)] $\mathcal{L}$ is a composition of one or more of the following types of transformations:
		\begin{itemize}
			\item[(a)] $A\mapsto FAE$, where $F$ and $E$ are $2\times2$ diagonal matrices with positive diagonal entries;
			\item[(b)] $A\mapsto -A$;
			\item[(c)] $A\mapsto P_2A$, in which $P_2$ is an exchange matrix; and
			\item[(d)] $A\mapsto A^T$.
		\end{itemize}
	\end{itemize}
\end{theorem}
\begin{proof}
	That (2)$\implies$(1) is immediate. We now show (1)$\implies$(2). By Lemma~\ref{lemma_preserver} we have $P(\mathcal{SSR})\subseteq P(\mathcal{SR})$, so any linear preserver of the set of $2\times2$ SSR matrices must be as in Theorem~\ref{thrmSR_1}(3). To avoid ambiguity, we relabel the maps (a)--(e) in Theorem~\ref{thrmSR_1}(3) as (a$^{\prime}$)--(e$^{\prime}$), respectively. We need to show that the linear preservers of the class of $2\times2$ SSR matrices are precisely the compositions of the maps (a$^{\prime}$)--(d$^{\prime}$), with the additional requirement that the matrix $H$ in (a$^{\prime}$) must be singular.
	
    Let $\mathcal{T}:\mathbb{R}^{2\times2}\to\mathbb{R}^{2\times2}$ be an arbitrary composition of the maps (a$^{\prime}$) through (e$^{\prime}$). Note that if $T_a$ is any map (a$^{\prime}$) and $T'$ is any map (b$^{\prime}$)--(e$^{\prime}$), then $T' T_a = T'_a T'$ for some other map $T'_a$ of type (a$^{\prime}$). Thus, all maps $T_a$ can be grouped at the leftmost position in the composition and collectively replaced by a single entrywise positive matrix. Furthermore, the map $A\mapsto-A$ preserves SSR matrices and commutes with all other transformations. Therefore, without loss of generality, we can assume that $\mathcal{T}(A):=H\circ\mathcal{B}(A)$ for all $A\in\mathbb{R}^{2\times2}$, where $H=\begin{pmatrix}
		l_1 & l_3 \\
		l_4 & l_2
	\end{pmatrix}$ is an entrywise positive matrix, and $\mathcal{B}$ is an arbitrary composition of maps selected from (c$^{\prime}$) to (e$^{\prime}$). We next show that $\mathcal{T}$ preserves SSR matrices if and only if $H$ is singular and $\mathcal{B}$ is composed only of maps (c$^{\prime}$) and (d$^{\prime}$).
	
	To prove the converse, suppose $H$ is singular and $\mathcal{B}$ is composed only of maps (c$^{\prime}$) and (d$^{\prime}$). Let $A$ be an SSR matrix. By simple calculations, one can verify that $\mathcal{T}(A)$ is SSR if and only if $A$ is SSR. Since $l_1l_2=l_3l_4$, one can express $\mathcal{T}(A)$ as
	\begin{align*}
		\begin{pmatrix}
			l_1 & 0 \\
			0 & l_4
		\end{pmatrix}\mathcal{B}(A)\begin{pmatrix}
			1 & 0 \\
			0 & l_2/l_4
		\end{pmatrix},
	\end{align*}
	which is a composition of a positive diagonal equivalence and the map $\mathcal{B}$.
	
	Next, we prove the forward implication by the contrapositive method. Suppose $\mathcal{T}$ preserves SSR matrices, but either $H$ is invertible or $\mathcal{B}$ involves the map (e$^{\prime}$). We show that $\mathcal{T}$ fails to preserve SSR in each such case. Note that $\mathcal{B}$ only permutes the entries of $2\times2$ matrices in one of the twenty-four ways listed in the middle column of Table~\ref{Ttable}. If $H$ is invertible, take $X=\begin{pmatrix}
		l_3l_4/l_1l_2 & 1 \\
		1 & 1
	\end{pmatrix}$ and define $A:=\mathcal{B}^{-1}(X)$. One can verify that $A$ is SSR, but $\det\mathcal{T}(A)=0$. Let $H$ be singular. We next claim that if $\mathcal{B}$ involves the map (e$^{\prime}$), then $\mathcal{T}$ fails to preserve the set of SSR matrices. Note that if $\mathcal{B}$ corresponds to a configuration in Table~\ref{Ttable}(I), then it is composed only of maps (c$^{\prime}$) and/or (d$^{\prime}$). Hence, we need to examine the following two cases.
	
	\begin{itemize}[leftmargin=0pt,label={}]
		\item \textbf{Case 1.} $\mathcal{B}$ corresponds to a configuration in Table~\ref{Ttable}(II): $\mathcal{B}$ consists of a map $\mathcal{P}_1$ on the left, which is a composition of maps (c$^{\prime}$) and/or (d$^{\prime}$), followed by a single application of map (e$^{\prime}$). Therefore, $\mathcal{T}(A)=H\circ$($\mathcal{P}_1$(e$^{\prime}$)$(A))$. Taking $A=\begin{pmatrix}
			1 & 2 \\
			2 & 1
		\end{pmatrix}$, which is SSR, yields $\det \mathcal{T}(A)=0$.
		
		\item \textbf{Case 2.} $\mathcal{B}$ corresponds to a configuration in Table~\ref{Ttable}(III): $\mathcal{B}$ consists of $\mathcal{P}_2$ on the left, which is a composition of maps (c$^{\prime}$) and/or (d$^{\prime}$), followed by (e$^{\prime}$), and then by (d$^{\prime}$). Thus, $\mathcal{T}(A) = H\circ$($\mathcal{P}_2$(e$^{\prime}$)(d$^{\prime}$)$(A))$. Using the same $A$ from case~1 gives $\det\mathcal{T}(A)=0$. This completes the proof. \qedhere
	\end{itemize}
\end{proof}

\section{Theorem~\ref{F}: Linear preserver problem for sign regularity with given sign pattern} \label{section SR_epsilon}

In the final section, we show Theorem~\ref{F}, i.e., we characterize all linear maps preserving sign regularity with a given sign pattern. Again, that (3)$\implies$(1) and (3)$\implies$(2) can be verified directly. We first show (2)$\implies$(3), and then end by proving (1)$\implies$(3). To prove (2)$\implies$(3), we must handle the cases $m=n$ and $m\neq n$ separately, since $A\mapsto A^T$ does not preserve $m\times n$ matrices.

Let $\mathcal{L}:\mathbb{R}^{m\times n}\to\mathbb{R}^{m\times n}$ for $m,n\geq2$ be a linear transformation such that $\mathcal{L}$ maps $\mathcal{SR}_2(\epsilon)$ onto itself. If $m\neq n$, then we assume without loss of generality that $m>n$. The case $m<n$ follows similarly. Also, assume without loss of generality that $\epsilon_1>0.$ For $\epsilon_1<0$, one can take the negatives of the  basis elements below and proceed similarly.

Note that $E_{ij}\in \mathcal{SR}_2(\epsilon)$ for all $1\leq i\leq m$, $1\leq j\leq n$; thus $\mathcal{L}^{-1}$ exists and further $\mathcal{L}^{-1}\in P(\mathcal{SR}_2(\epsilon))$. Let $\mathfrak{B}$ be the ordered basis~\eqref{basis} and $L$ be the matrix which represents the linear transformation $\mathcal{L}$ with respect to this basis.
\begin{prop} \label{SR_epsilon_monomial}
	$L$ is a monomial matrix.
\end{prop}	
\begin{proof}
	Since $\mathcal{L}, \mathcal{L}^{-1}\in P(\mathcal{SR}_2(\epsilon))$, $L$ and $L^{-1}$ are entrywise non-negative.
\end{proof}	
Note that Remark~\ref{Sij_singleton} holds and all the entries of $\mathcal{L}(J)=Y$ are positive.
\begin{prop} \label{SR_epsilon_S11Snn_S1nSn1_combo}
	For $\mathcal{L}\in P(\mathcal{SR}_2(\epsilon))$ with $\epsilon=(\epsilon_1,\epsilon_2)$, the following pairwise combinations are possible.
	\begin{itemize}
		\item[(i)] $S_{11}=\{E_{11}\}~ \text{and} ~ S_{mn}=\{E_{mn}\}, ~ \text{or} ~ S_{11}=\{E_{mn}\} ~ \text{and} ~ S_{mn}=\{E_{11}\}$.
		
		\item[(ii)] $S_{1n}=\{E_{1n}\} ~ \text{and} ~ S_{m1}=\{E_{m1}\}, ~ \text{or} ~ S_{1n}=\{E_{m1}\} ~ \text{and} ~ S_{m1}=\{E_{1n}\}$.
	\end{itemize}
\end{prop}
\begin{proof}
	Let us begin by proving (i). To the contrary, suppose $S_{11}\neq\{E_{11}\}, \{E_{mn}\}$. Let $J(c)$ be the matrix obtained by multiplying the $(1,1)$ entry of $J=J_{m\times n}$ by $c>0$. Then $J(c)\in \mathcal{SR}_2(\epsilon)$ with $\epsilon_1=\epsilon_2=1$ for $c>1$, and with $\epsilon_1=1$, $\epsilon_2=-1$ for $0<c<1$. Thus, $\mathcal{L}(J(c)):=Y(c)\in \mathcal{SR}_2(\epsilon)$. Now, consider the following cases.
	
	If $S_{11}=\{E_{1k}\}$ for $k\neq 1$, an appropriate choice of $c>0$ gives us $\epsilon_2\det\begin{pmatrix}
		y_{11} & cy_{1k} \\
		y_{m1} & y_{mk}
	\end{pmatrix}<0$, a contradiction. Similarly, we can show that $S_{11}\neq\{E_{k1}\}$ for $k\neq1$, $S_{11}\neq\{E_{mk}\}$ for $k\neq n$, and $S_{11}\neq\{E_{kn}\}$ for $k\neq m$. 
	
	If $S_{11}=\{E_{ij}\}$ where $(i,j)\neq(1,1)$ and $(m,n)$, then an appropriate choice of $c>0$ gives us $\epsilon_2\det\begin{pmatrix}
		y_{i1} & cy_{ij} \\
		y_{m1} & y_{mj}
	\end{pmatrix}<0$, a contradiction. Hence, either $S_{11}=\{E_{11}\}~ \text{or} ~ S_{11}=\{E_{mn}\}$. \smallskip
	
	Again, to the contrary, suppose that $S_{mn}\neq\{E_{11}\}, \{E_{mn}\}$. In this case, let $J(c)$ be the matrix obtained by multiplying the $(m,n)$ entry of $J=J_{m\times n}$ by $c>0$. Proceed similar to the previous case to show that our assumption is not true. Hence, $S_{mn}=\{E_{11}\}$ or
	$S_{mn}=\{E_{mn}\}$. By Proposition~\ref{SR_epsilon_monomial}, for any $\epsilon=(\epsilon_1,\epsilon_2)$, we have either 
	\[S_{11}=\{E_{11}\}~ \text{and} ~ S_{mn}=\{E_{mn}\}, ~ \text{or} ~ S_{11}=\{E_{mn}\} ~ \text{and} ~ S_{mn}=\{E_{11}\}.\] 
	The same argument can now be adapted as in the preceding half of this proof to show~(ii).
\end{proof}

\begin{prop} \label{SR_epsilon_FirstLast_Row_Col}
	Let $\mathcal{L}\in P(\mathcal{SR}_2(\epsilon))$ with $\epsilon=(\epsilon_1,\epsilon_2)$ such that
	\begin{equation}\label{SR_epsilon_S11SnnSn1S1n_wlog}
		S_{11}=\{E_{11}\},~ S_{mn}=\{E_{mn}\}, ~ S_{1n}=\{E_{1n}\}, ~ \text{and} ~ S_{m1}=\{E_{m1}\}.
	\end{equation}
	\begin{itemize}
		\item [(i)] Then $\mathcal{L}$ maps the first (and last) row and column of its arguments entirely to the first (and last) row and column, respectively.
		
		\item [(ii)] $\mathcal{L}$ maps each row and column of its arguments entirely to some row and column, respectively.
	\end{itemize} 
\end{prop}
\begin{proof}
	First, we show that $\mathcal{L}$ must map the first row of its arguments entirely to the first row. This holds trivially for $n=2$ by~\eqref{SR_epsilon_S11SnnSn1S1n_wlog}. Therefore, let $n\geq3$. Let $J(c)$ be the matrix obtained by multiplying the first row of $J=J_{m\times n}$ by $c>0$. Then \[J(c)\in \mathcal{SR}_2(\epsilon)~ \text{with}~ \epsilon_1=1\quad \implies \quad \mathcal{L}(J(c)):=Y(c)\in \mathcal{SR}_2(\epsilon)~\text{with}~ \epsilon_1=1. \] 
	Assume $\mathcal{L}$ does not map the first row of its arguments entirely to the first row. Thus, there exists $1<p<n$ such that the $(1,p)$ position of the matrix $Y(c)$ is not occupied by the image of any element from the first row of $J(c)$. Using~\eqref{SR_epsilon_S11SnnSn1S1n_wlog} and an appropriate choice of $c>0$ gives
	\[\epsilon_2\det\begin{pmatrix}
		cy_{11} & y_{1p} \\
		y_{m1} & \alpha y_{mp}
	\end{pmatrix}=\epsilon_2(c\alpha y_{11}y_{mp}-y_{1p}y_{m1})<0,\]
	where either $\alpha=c$ or $\alpha=1$, a contradiction. Similarly, we can prove the other cases by multiplying that particular row and column of $J=J_{m\times n}$ by $c>0$. This shows part (i).
	
	To prove part (ii), we will first show that $\mathcal{L}$ must map the $k$th row of its arguments entirely to some row, where $1<k<m$. Let $J(c)$ be the matrix obtained by multiplying the $k$th row of $J=J_{m\times n}$ by $c> 0$. By Proposition~\ref{SR_epsilon_FirstLast_Row_Col}(i), $\mathcal{L}$ maps the $(k,1)$ element of $J(c)$ to the first column; let us say $\mathcal{L}$ maps it to the $(s,1)$ position of $Y(c)$, where $1<s<m$. If $\mathcal{L}$ does not map the entire $k$th row of $J(c)$ to the $s$th row of $Y(c)$, then there exists $1<j\leq n$ such that the $(s,j)$ position of the matrix $Y(c)$ is not occupied by the image of any element from the $k$th row of $J(c)$. Using Proposition~\ref{SR_epsilon_FirstLast_Row_Col}(i) and an appropriate choice of $c>0$ gives $\epsilon_2\det\begin{pmatrix}
		cy_{s1} & y_{sj} \\
		y_{m1} & y_{mj}
	\end{pmatrix}<0$, a contradiction.
	Similarly, we can prove this for columns. 
\end{proof}
Using the above propositions, we now prove Theorem~\ref{F} for $m=n\geq2$.

\begin{proof}[Proof of Theorem~\ref{F}~(2)$\implies$(3) for $m=n$] 
	Let $n\geq2$ and $\mathcal{L}:\mathbb{R}^{n\times n}\to\mathbb{R}^{n\times n}$ be a linear transformation such that $\mathcal{L}$ maps $\mathcal{SR}_2(\epsilon)$ onto itself. By Proposition~\ref{SR_epsilon_S11Snn_S1nSn1_combo}, we can assume without loss of generality, that
	\begin{equation}\label{SR_epsilon_square_S11SnnSn1S1n_wlog}
		S_{11}=\{E_{11}\}, ~ S_{nn}=\{E_{nn}\}, ~ S_{1n}=\{E_{1n}\}, ~ \text{and} ~ S_{n1}=\{E_{n1}\}
	\end{equation}
	since $\mathcal{L}\in P(\mathcal{SR}(\epsilon))$ if and only if $\mathcal{L}$ composed with the maps $A\mapsto P_nAP_n$ and $A\mapsto A^T$ is a linear SR$(\epsilon)$-preserver. Hence, Proposition~\ref{SR_epsilon_FirstLast_Row_Col} holds.
	
	We now complete the proof by induction on $n$ with the
	base case $n=2$. Let $\mathcal{L}:\mathbb{R}^{2\times2}\to\mathbb{R}^{2\times2}$ be a linear map such that $\mathcal{L}(\mathcal{SR}_2(\epsilon))=\mathcal{SR}_2(\epsilon)$ and let $\mathfrak{B}=(E_{11}, E_{22}, E_{12}, E_{21})$ be an ordered basis of $\mathbb{R}^{2\times 2}$. By Propositions~\ref{SR_epsilon_monomial}~and~\ref{SR_epsilon_FirstLast_Row_Col}, we have $S_{ij}=\{E_{ij}\}\ \forall~1\leq i,j\leq 2$. Thus, $\mathcal{L}(E_{11})=l_{1}E_{11},\ldots,\mathcal{L}(E_{21})=l_{4}E_{21}\in
	\mathcal{SR}_2(\epsilon)$ with $\epsilon_1=1$. By Proposition~\ref{SR_epsilon_monomial}, $l_{i}>0\ \forall~1\leq i\leq 4$. 

	Again, since $J=J_{2 \times2}
	\in \mathcal{SR}_2(\epsilon) ~\text{with}~ \epsilon_1=1$,
	we have 
	\[\mathcal{L}(J)=\begin{pmatrix}
		l_{1} & l_{3} \\
		l_{4} & l_{2}
	\end{pmatrix}\in \mathcal{SR}_2(\epsilon)~\text{and}~ \mathcal{L}^{-1}(J)=\begin{pmatrix}
		1/l_1 & 1/l_3 \\
		1/l_4 & 1/l_2
	\end{pmatrix}\in \mathcal{SR}_2(\epsilon) \]
	and hence $l_{1}l_{2} = l_{3}l_{4}$. Let $B=\begin{pmatrix}
		b_{11} & b_{12} \\
		b_{21} & b_{22} 
	\end{pmatrix}\in \mathcal{SR}_2(\epsilon)$. Then
	\[ \mathcal{L}(B)=\begin{pmatrix}
		l_{1}b_{11} & l_{3}b_{12} \\
		l_{4}b_{21} & l_{2}b_{22}
	\end{pmatrix} = \begin{pmatrix}
		l_{1} & 0  \\
		0 & l_{4}
	\end{pmatrix}\begin{pmatrix}
		b_{11} & b_{12} \\
		b_{21} & b_{22} 
	\end{pmatrix}\begin{pmatrix}
		1 & 0  \\
		0 & l_{3}/l_{1}
	\end{pmatrix},\]
	which is a positive diagonal equivalence. This completes the proof for $n=2$.
	
	For the induction step, let $\mathcal{L}:\mathbb{R}^{n\times n} \to\mathbb{R}^{n\times n}$ be a linear $\mathcal{SR}_2(\epsilon)$-preserver with $n>2$. Let $A\in\mathcal{SR}_2(\epsilon)$. By Proposition~\ref{SR_epsilon_FirstLast_Row_Col}(i), the leading principal submatrix of $A$ of size $(n-1)\times (n-1)$ must be transformed to the leading principal submatrix of size $(n-1)\times (n-1)$ of $\mathcal{L}(A)$. Since every SR$_2(\epsilon)$ matrix $\widehat{A}$ of size
	$(n-1)\times(n-1)$ is a leading principal submatrix of the SR$_2(\epsilon)$ matrix $\widehat{A}\oplus(0)_{1\times1}$ of size $n\times n$, the natural restriction of $\mathcal{L}$ onto the $(n-1)\times(n-1)$ leading
	principal submatrix is a linear $\mathcal{SR}_2(\epsilon)$-preserver on
	$\mathbb{R}^{(n-1)\times(n-1)}$. By the induction hypothesis, it is a composition of one or more of the following maps: (i)~$X\mapsto P_{n-1}XP_{n-1}$, (ii)~$X\mapsto X^T$, and (iii)~$X\mapsto FXE$. In fact, it is a positive diagonal equivalence since the first row is transformed to the first row. Thus $S_{ij}=\{E_{ij}\}\ \forall~1\leq i,j\leq
	n-1.$ By Proposition~\ref{SR_epsilon_FirstLast_Row_Col}, $\mathcal{L}$ maps each row and column of its arguments entirely to some row and column, respectively, and so
	\begin{equation}\label{SR_epsilon_square_Sij=Eij}
		S_{ij}=\{E_{ij}\}~\text{for all} ~ 1\leq i,j\leq n.
	\end{equation}
	Since we may compose the inverse positive diagonal equivalence relative to the upper left principal submatrix of size $(n-1)\times(n-1)$ with $\mathcal{L}$, we assume without loss of generality that $A$ and $\mathcal{L}(A)$ have the same leading principal $(n-1) \times (n-1)$ submatrix. Using~\eqref{SR_epsilon_square_Sij=Eij}, we have
	\[\mathcal{L}(E_{in})=c_iE_{in}, ~ \mathcal{L}(E_{ni})=k_iE_{ni}, ~\text{for}~ 1\leq i\leq n-1, ~\text{and}~
	\mathcal{L}(E_{nn})=dE_{nn}\] 
	for some positive scalars $c_i,k_i$, and $d$. We next claim that
	\begin{equation}\label{eqthrmE3}
		c_1=\cdots=c_{n-1}, ~k_1=\cdots=k_{n-1},~\text{and}~d=c_1k_1.
	\end{equation}
	Let $J=J_{n\times n}$. Then
	\[\mathcal{L}(J)=\begin{pmatrix}
		1 & \cdots & 1 & c_1 \\
		\vdots & \ddots & \vdots & \vdots \\
		1 & \cdots & 1 & c_{n-1} \\
		k_1 & \cdots & k_{n-1} & d
	\end{pmatrix}\in \mathcal{SR}_2(\epsilon) ~\text{and}~\mathcal{L}^{-1}(J)=\begin{pmatrix}
		1 & \cdots & 1 & 1/c_1 \\
		\vdots & \ddots & \vdots & \vdots \\
		1 & \cdots & 1 & 1/c_{n-1} \\
		1/k_1 & \cdots & 1/k_{n-1} & 1/d
	\end{pmatrix}\in \mathcal{SR}_2(\epsilon).\]
	Thus,~\eqref{eqthrmE3} holds. Hence, $\mathcal{L}$ maps $A$ to $FAE$ for some positive diagonal matrices $F$ and $E$. This concludes the induction step and the proof.
\end{proof}

Next, we prove Theorem~\ref{F} for $m>n\geq2$, using the propositions at the beginning of this section.
\begin{proof}[Proof of Theorem~\ref{F}~(2)$\implies(3)$ for $m\neq n$] 
	Let $\mathcal{L}:\mathbb{R}^{m\times n}\to\mathbb{R}^{m\times n}$ where $m>n\geq2$ be a linear transformation such that $\mathcal{L}$ maps $\mathcal{SR}_2(\epsilon)$ onto itself. Since $\mathcal{L}$ is an $\mathcal{SR}_2(\epsilon)$-preserver if and only if $\mathcal{L}$ composed with the map $A\mapsto P_{m}AP_{n}$ is also one, by Proposition~\ref{SR_epsilon_S11Snn_S1nSn1_combo}(i), we can assume without loss of generality that 
	\begin{equation}\label{SR_epsilon_rectangle_S11Smn_wlog}
		S_{11}=\{E_{11}\}~ \text{and} ~ S_{mn}=\{E_{mn}\}.
	\end{equation} 
	Therefore, from~\eqref{SR_epsilon_rectangle_S11Smn_wlog} and Proposition~\ref{SR_epsilon_S11Snn_S1nSn1_combo}(ii), we have either
	\begin{align*}
		&S_{11}=\{E_{11}\}, ~ S_{mn}=\{E_{mn}\}, ~ S_{1n}=\{E_{1n}\}, ~ \text{and} ~S_{m1}=\{E_{m1}\}, ~ \text{or} \\
		&S_{11}=\{E_{11}\}, ~ S_{mn}=\{E_{mn}\}, ~ S_{1n}=\{E_{m1}\}, ~ \text{and} ~ S_{m1}=\{E_{1n}\}.
	\end{align*}
	Next, we show that the conditions $S_{11}=\{E_{11}\}$, $S_{mn}=\{E_{mn}\}$, $S_{1n}=\{E_{m1}\}$, $S_{m1}=\{E_{1n}\}$ do not simultaneously hold. Assume to the contrary that they do. Let $J(c)$ be the matrix obtained by multiplying the first row of $J=J_{m\times n}$ by $c>0$. Then $J(c)\in \mathcal{SR}_2(\epsilon)$ and hence $\mathcal{L}(J(c)):=Y(c)\in \mathcal{SR}_2(\epsilon)$ with $\epsilon_1=1$. Since $m>n$, even if $\mathcal{L}$ maps all elements of the first row of $J(c)$ to the first column of $Y(c)$, there exists $1<j<m$ such that the $(j,1)$ position of $Y(c)$ is not occupied by the image of any element from the first row of $J(c)$. Therefore, by our assumption and an appropriate choice of $c>0$, we obtain $\epsilon_2\det\begin{pmatrix}
		cy_{11} & y_{1n} \\
		y_{j1} & \alpha y_{jn}
	\end{pmatrix}<0,$ where either $\alpha=c$ or  $\alpha=1$, a contradiction. Thus
	\begin{equation} \label{SR_epsilon_rectangle_S11SmnS1nSm1_wlog}
		S_{11}=\{E_{11}\},~ S_{mn}=\{E_{mn}\},~ S_{1n}=\{E_{1n}\}, ~\text{and}~ S_{m1}=\{E_{m1}\},
	\end{equation}
	and Proposition~\ref{SR_epsilon_FirstLast_Row_Col} holds. We now complete the proof by induction on the size of the matrices, with the base case $m=3$ and $n=2$. Let $\mathcal{L}:\mathbb{R}^{3\times2}\to\mathbb{R}^{3\times2}$ such that $\mathcal{L}(\mathcal{SR}_2(\epsilon))=\mathcal{SR}_2(\epsilon)$. 
	
	Let $\mathfrak{B}=(E_{11}, E_{22}, E_{12}, E_{21},E_{31},E_{32})$ be an ordered basis of $\mathbb{R}^{3\times 2}$. By Proposition~\ref{SR_epsilon_FirstLast_Row_Col}, we have $S_{ij}=\{E_{ij}\}$ for all $1\leq i\leq 3$,~ $1\leq j\leq2$. Thus, $\mathcal{L}(E_{11})=l_{1}E_{11},\ldots,\mathcal{L}(E_{32})=l_{6}E_{32}\in \mathcal{SR}_2(\epsilon)$ with $\epsilon_1=1$. By Proposition~\ref{SR_epsilon_monomial}, $l_{i}>0$ for all $1\leq i\leq 6$.
	
	Since $A_1:=\begin{pmatrix}
		1 & 1 \\
		1 & 1 \\
		0 & 0
	\end{pmatrix}, \ A_2:=\begin{pmatrix}
		1 & 1 \\
		0 & 0 \\
		1 & 1
	\end{pmatrix} \in \mathcal{SR}_2(\epsilon),$ hence the four matrices
	\[
	\mathcal{L}(A_1) = \begin{pmatrix}
		l_{1} & l_{3} \\
		l_{4} & l_{2} \\
		0 & 0
	\end{pmatrix}, \quad \mathcal{L}^{-1}(A_1)=\begin{pmatrix}
		1/l_1 & 1/l_3 \\
		1/l_4 & 1/l_2 \\
		0 & 0
	\end{pmatrix}, \quad
	\mathcal{L}(A_2) = \begin{pmatrix}
		l_{1} & l_{3} \\
		0 & 0 \\
		l_{5} & l_{6}
	\end{pmatrix}, \quad \mathcal{L}^{-1}(A_2)=\begin{pmatrix}
		1/l_1 & 1/l_3 \\
		0 & 0   \\
		1/l_5 & 1/l_6
	\end{pmatrix}
	\]
	are all in $\mathcal{SR}_2(\epsilon)$. It follows that $l_{1}l_{2}=l_{3}l_{4}$ and $l_{1}l_{6}=l_{3}l_{5}$. Let $B=\begin{pmatrix}
			b_{11} & b_{12} \\
			b_{21} & b_{22} \\
			b_{31} & b_{32}
		\end{pmatrix}\in \mathcal{SR}_2(\epsilon)$. Then
	\[\mathcal{L}(B)=\begin{pmatrix}
		l_{1}b_{11} & l_{3}b_{12} \\
		l_{4}b_{21} & l_{2}b_{22} \\
		l_{5}b_{31} & l_{6}b_{32}
	\end{pmatrix} = \begin{pmatrix}
		l_{1} & 0 & 0 \\
		0 & l_{4} & 0 \\
		0 & 0 & l_{5}
	\end{pmatrix}\begin{pmatrix}
		b_{11} & b_{12} \\
		b_{21} & b_{22} \\
		b_{31} & b_{32}
	\end{pmatrix}\begin{pmatrix}
		1 & 0  \\
		0 & l_{3}/l_{1}
	\end{pmatrix},\] 
	which is a positive diagonal equivalence. This completes the proof for the base case.
	
	For the induction step, let $\mathcal{L}:\mathbb{R}^{m\times2}\to\mathbb{R}^{m\times2}$ be a linear map such that $\mathcal{L}\in P(\mathcal{SR}_2(\epsilon))$. Let $A\in\mathcal{SR}_2(\epsilon)$ of size $m\times2$. By Proposition~\ref{SR_epsilon_FirstLast_Row_Col}(i), the submatrix of $A$ formed by the first $(m-1)$ rows and both columns must be transformed to the first $(m-1)$ rows of $\mathcal{L}(A)$. Since every SR$_2(\epsilon)$ matrix $\widehat{A}$ of size $(m-1)\times 2$ is a submatrix of the SR$_2(\epsilon)$ matrix $(\widehat{A}^T|\mathbf{0})^T\in\mathbb{R}^{m\times 2}$, the natural restriction of $\mathcal{L}$ onto the $(m-1)\times 2$ top submatrix  is a linear $\mathcal{SR}_2(\epsilon)$-preserver on $\mathbb{R}^{(m-1)\times 2}$. By the induction hypothesis, the restriction of $\mathcal{L}$ is a composition of one or more of the following maps: (i)~$X\mapsto P_{m-1}XP_2$, and (ii)~$X\mapsto FXE$. In fact, it is a positive diagonal equivalence since the first row and column is transformed to the first row and column, respectively. Thus, $S_{ij}=\{E_{ij}\}$ for all $1\leq i\leq m-1$ and $1\leq j\leq 2$. By Proposition~\ref{SR_epsilon_FirstLast_Row_Col}(i), $\mathcal{L}$ maps the first (and last) column of its arguments entirely to the first (and last) column, so
	\begin{equation}\label{SR_epsilon_rectangle_Sij=Eij}
		S_{ij}=\{E_{ij}\} ~\text{for all}~ 1\leq i\leq m,~  1\leq j\leq 2.
	\end{equation} 
	Since we may compose the inverse positive diagonal equivalence relative to the upper submatrix of size $(m-1)\times 2$ with $\mathcal{L}$, we may assume without loss of generality that all but the last row of $A$ and $\mathcal{L}(A)$ are equal. Using~\eqref{SR_epsilon_rectangle_Sij=Eij}, we have $\mathcal{L}(E_{mi})=k_iE_{mi}$ for $1\leq i\leq 2$, for some positive scalars $k_1$ and $k_2$. We next claim that $k_1=k_{2}.$ Let $J=J_{m\times2}$. Then
	\[\mathcal{L}(J)=\begin{pmatrix}
		J_{(m-1) \times 2}\\ k_1 \qquad k_2
	\end{pmatrix} \in \mathcal{SR}_2(\epsilon)
	~\text{and} ~
	\mathcal{L}^{-1}(J)= \begin{pmatrix}
		J_{(m-1) \times 2}\\ 1/k_1 \qquad
		1/k_2
	\end{pmatrix}\in \mathcal{SR}_2(\epsilon). \]
	Therefore, we have $k_1=k_{2}$ and hence $\mathcal{L}$ is a positive diagonal equivalence. This completes the induction step for $m$ and the result holds for all linear $\mathcal{SR}_2(\epsilon)$-preservers from $\mathbb{R}^{m\times2}$ to $\mathbb{R}^{m\times2}$. 
	
	Now, fix arbitrary $m$ and suppose the claim holds for all linear $\mathcal{SR}_2(\epsilon)$-preservers from $\mathbb{R}^{m\times(n-1)}\to\mathbb{R}^{m\times(n-1)}$. Let $\mathcal{L}:\mathbb{R}^{m\times n}\to\mathbb{R}^{m\times n}$ be a linear $\mathcal{SR}_2(\epsilon)$-preserver. For $A\in\mathcal{SR}_2(\epsilon)$ of size $m\times n$, the submatrix of $A$ formed by all rows and the first $(n-1)$ columns must be transformed to the first $(n-1)$ columns of $\mathcal{L}(A)$ by Proposition~\ref{SR_epsilon_FirstLast_Row_Col}(i). Since every SR$_2(\epsilon)$ matrix $\widehat{A}$ of size $m\times(n-1)$ is a submatrix of the SR$_2(\epsilon)$ matrix $(\widehat{A}|\mathbf{0})\in\mathbb{R}^{m\times n}$, the natural restriction of $\mathcal{L}$ to the $m\times(n-1)$ left submatrix is a linear $\mathcal{SR}_2(\epsilon)$-preserver on $\mathbb{R}^{m\times(n-1)}$. By the induction hypothesis, it is a composition of one or more of the maps:
	(i)~$X\mapsto P_mXP_{n-1}$, and (ii)~$X\mapsto FXE$. By the same argument as in the preceding part,
	\[S_{ij}=\{E_{ij}\} ~\text{for all}~ 1\leq i\leq m,~ 1\leq j\leq n-1.\] 
	By Proposition~\ref{SR_epsilon_FirstLast_Row_Col}, $\mathcal{L}$ maps each row of its arguments entirely to some row and hence
	\begin{equation}\label{SR_epsilon_rectangle_second_Sij=Eij}
		S_{ij}=\{E_{ij}\} ~\text{for all}~ 1\leq i\leq m,~ 1\leq j\leq n.
	\end{equation} 
	Since we may compose the inverse positive diagonal equivalence relative to the left submatrix of size $m\times(n-1)$ with $\mathcal{L}$, we may assume without loss of generality that all but the last column of $A$ and $\mathcal{L}(A)$ are equal. Using~\eqref{SR_epsilon_rectangle_second_Sij=Eij}, we have
	\[\mathcal{L}(E_{in})=c_iE_{in}~\text{for}~ 1\leq i\leq m, ~\text{for some positive scalar} ~c_i.\]
	Now, to complete the proof, we must show that $c_1=\cdots=c_{m}.$ Let $J=J_{m\times n}$. Then
	\[\mathcal{L}(J)=\begin{pmatrix}
		1 & \cdots & 1 & c_1 \\
		1 & \cdots & 1 & c_2 \\
		\vdots & \ddots & \vdots & \vdots \\
		1 &  \cdots & 1 & c_{m} 
	\end{pmatrix}\in \mathcal{SR}_2(\epsilon) ~\text{and}~ \mathcal{L}^{-1}(J)=\begin{pmatrix}
		1 & \cdots & 1 & 1/c_1 \\
		1 & \cdots & 1 & 1/c_2 \\
		\vdots & \ddots & \vdots & \vdots \\
		1 & \cdots & 1 & 1/c_m 
	\end{pmatrix}\in \mathcal{SR}_2(\epsilon). \]
	Therefore, we have $c_1=\cdots=c_{m}$. Thus, $\mathcal{L}$ maps $A$ to $FAE$ for some positive diagonal matrices $F$ and $E$. This concludes the induction step.
\end{proof}

\begin{proof}[Proof of Theorem~\ref{F}~(1)$\implies(3)$] 
	We show (1)$\implies$(3) simultaneously for both cases: $m=n$ and $m\neq n$. Note that in the proof of (2)$\implies$(3) above, all test matrices $A$ considered were SR$(\epsilon)$, and the argument involved only the $1\times1$ and $2\times2$ minors of the output matrices $\mathcal{L}(A)$. Therefore, assuming (1) instead of (2), the same steps listed above give the preservers listed in (3). This concludes the proof.
\end{proof}

\begin{rem}
	Theorem~\ref{F} holds also for the class of SSR$(\epsilon)$ matrices. The proof follows Remark~\ref{SSR_proof} verbatim.
\end{rem}	
	
	\section*{Acknowledgments}
	We thank Apoorva Khare and Gadadhar Misra for a detailed reading of an earlier draft and for providing valuable feedback. We are also grateful to the anonymous referee for carefully going through the manuscript and offering several constructive comments that helped improve the exposition. The first author was partially supported by ANRF Prime Minister Early Career Research Grant ANRF/ECRG/2024/002674/PMS (ANRF, Govt.~of India), INSPIRE Faculty Fellowship Research Grant DST/INSPIRE/04/2021/002620 (DST, Govt.~of India), and IIT Gandhinagar Internal Project: IP/IP/50025.

\end{document}